\documentclass[11pt, oneside]{amsart} 

\usepackage[english]{babel} 
\usepackage{graphics,color,pgf}
\usepackage{epsfig}
\usepackage[all]{xy}
\usepackage{tikz-cd}
\usepackage{bbm}
\usetikzlibrary{matrix,arrows,decorations.pathmorphing}
\newdir{ >}{!/8pt/@{}*@{>}}
\usepackage{amssymb, amsmath,amsthm, mathtools, amscd}
\usepackage{mathrsfs}
\usepackage[left=3.5cm,right=3.5cm,top=2.5cm,bottom=2.5cm]{geometry}
\usepackage{tikz}
\usetikzlibrary{automata, arrows.meta, positioning}
\usepackage[pagebackref, pdfborder={0 0 0}
  ,urlcolor=black,hypertexnames=false]{hyperref}

\makeatletter
\newtheorem*{rep@theorem}{\rep@title}
\newcommand{\newreptheorem}[2]{%
\newenvironment{rep#1}[1]{%
 \def\rep@title{#2 \ref{##1}}%
 \begin{rep@theorem}}%
 {\end{rep@theorem}}}
\makeatother

\newtheorem{intro_thm}{Theorem}
\newtheorem{intro_cor}[intro_thm]{Corollary}

\newtheorem{intro_defn}[intro_thm]{Definition}

\newtheorem{lemma}{Lemma}[section]
\newtheorem{thm}[lemma]{Theorem} 
\newtheorem{prop}[lemma]{Proposition}
\newtheorem{cor}[lemma]{Corollary}

\theoremstyle{definition}
\newtheorem{defn}[lemma]{Definition}
\newtheorem{es}[lemma]{Example}

\newtheorem{notation}[lemma]{Notation}

\theoremstyle{remark}
\newtheorem{oss}[lemma]{Remark}

\newtheoremstyle{TheoremNum}
        {0.2 cm}{0.2 cm}              
        {\itshape}                      
        {}                              
        {}                     
        {.}                             
        { }                             
        {\thmname{\bfseries #1}\thmnote{ \bfseries #3}}
    \theoremstyle{TheoremNum}
   
\newtheorem{rec_thm}{Theorem}

\newtheorem{rec_cor}[rec_thm]{Corollary}
\newtheorem{rec_defn}[rec_thm]{Definition}

\newcommand{\id}{\mathrm{id}}

\newcommand\calQ{{\mathcal Q}}

\newcommand\calM{{\mathcal M}}
\newcommand\calN{{\mathcal N}}
\newcommand\calG{{\mathcal G}}

\newcommand\calE{{\mathcal E}}

\newcommand\calK{{\mathcal K}}

\newcommand\calF{{\mathcal F}}

\newcommand{\Hm}{\textup{H}}

\newcommand{\Hmb}{\textup{H}_{\textup{mb}}}

\newcommand{\Linf}{\mathrm{L}^{\infty}}

\newcommand{\Lone}{\mathrm{L}^{1}}

\newcommand{\Ima}{\text{Im}}
\newcommand{\Ker}{\text{Ker}}

\newcommand{\Id}{\text{Id}}

\newcommand{\Linfw}{\mathrm{L}_{\textup{w}^*}^{\infty}}

\begin{document}

\title[Bounded cohomology of measured groupoids]{Measurable bounded cohomology of $t$-discrete measured groupoids via resolutions}

\author[F. Sarti]{F. Sarti}
\address{Dipartimento di Matematica, Universit\'a di Pisa}
\email{filippo.sarti@unipi.it}

\author[A. Savini]{A. Savini}
\address{Dipartimento di Matematica, Universit\'a di Milano-Bicocca}
\email{alessio.savini@unimib.it}

\date{\today.\ \copyright{\ F. Sarti, A. Savini}}

\begin{abstract}
We define bounded cohomology of $t$-discrete measured groupoids with coefficients into measurable bundles of Banach spaces. Our approach via homological algebra extends the classic theory developed by Ivanov and by Monod. As a consequence, we show that the bounded cohomology of a $t$-discrete groupoid $\mathcal{G}$ can be computed using any amenable $\mathcal{G}$-space. In particular, we can compute bounded cohomology using strong boundaries.
\end{abstract}

\maketitle

\section{Introduction}

The \emph{continuous cohomology} $\mathrm{H}^\bullet_c(G)$ of a topological group $G$ with complex coefficients is defined by considering the cohomology of the complex of $G$-invariant $\mathbb{C}$-valued continuous functions on $G$.  There exist several and relevant cases in which the continuous cohomology is well understood. For instance, when $G$ is discrete, continuity is trivially satisfied and its cohomology boils down to the one of the associated classifying space $BG$ \cite[Chapter II.4]{brown}. When $G$ is a semisimple Lie group of non-compact type, we can recover its continuous cohomology either by looking at the $G$-invariant differential forms on the associated symmetric space \cite[Chapter IX]{Borel} or by considering the De Rham cohomology of the compact dual symmetric space \cite[Chapitre III.7]{guichardet}. Blanc \cite{Blanc} showed that if $G$ is $\sigma$-compact and it acts on a locally compact $\sigma$-compact space $X$ by fixing some positive measure on it, then the complex of locally $p$-integrable functions on $X$ still computes the continuous cohomology of $G$. More recently Austin and Moore \cite{AM} have proved that, if $G$ is locally compact and second countable, then its continuous cohomology is equivalent to its measurable variant. 

By focusing our attention on the subcomplex of \emph{bounded} continuous functions, we can construct the \emph{continuous bounded cohomology} $\mathrm{H}_{cb}^\bullet(G)$. The latter is a much more mysterious invariant and explicit computations are known only in few cases. For instance, when $G$ is \emph{amenable} its continuous bounded cohomology is trivial. For semisimple Lie groups, continuous bounded cohomology is well understood only in low degrees (see for instance Burger and Monod \cite{BM1}). 

The need to compute bounded cohomology using different resolutions pushed Ivanov \cite{Ivanov} to develop a new framework in the context of discrete groups which mimics the one already known for the usual cohomology. Indeed, a relevant part of his work was devoted to reformulate classic notions like \emph{injectivity} for modules and the definition of \emph{contracting homotopy} in the bounded setting. Thanks to this new perspective, he was able to show that the singular bounded cohomology of an aspherical space $X$ depends only on its fundamental group $\pi_1(X)$ \cite[Theorem 5.5]{miolibro}, in complete analogy to the unbounded case.

Almost twenty years later, Burger and Monod \cite{burger2:articolo} extended Ivanov's approach to continuous bounded cohomology of locally compact groups. One of the crucial aspects of their theory was the possibility to compute continuous bounded cohomology via the complex of essentially bounded measurable functions on any \emph{amenable} $G$-space \cite[Theorem 2]{burger2:articolo}. Relevant examples of amenable $G$-spaces are either the Furstenberg boundary when $G$ is a semisimple Lie group or the Poisson boundary with respect to an \'etale measure on $G$. 


Since groups are particular instances of \emph{groupoids} (\emph{i.e.} small categories in which every morphism is invertible), it is natural to look for any suitable cohomological theory for groupoids. When the groupoid is \emph{measured}, namely it admits a well-behaved measure under multiplication and inversion, Westman \cite{westman69} exploited the complex of measurable functions to extend measurable cohomology of groups to groupoids. Something similar was done by Feldman and Moore \cite{feldman:moore} in the case of countable equivalence relations and by Series for more general fibred coefficients \cite{Series}. 

With the goal of adding a missing piece to the puzzle described above and inspired by several rigidity results for measurable cocycles \cite{sarti:savini:3,sarti:savini:20:finite:reducibility,sarti:savini:2}, in a recent work \cite{sarti:savini:23} the authors settled the foundation of the theory of bounded cohomology for \emph{measured groupoids}. More precisely, given a measured groupoid $\mathcal{G}$, its bounded cohomology is defined by taking essentially bounded functions on the fibred space $\mathcal{G}^{(\bullet+1)}$ consisting of $(\bullet+1)$-tuples sharing the same target. 
The analogies with bounded cohomology of locally compact groups extends beyond the mere definition. Just to mention few of them, we recall the \emph{exponential isomorphism} \cite[Theorem 1]{sarti:savini:23} and the vanishing result for \emph{amenable groupoids} \cite[Theorem 4]{sarti:savini:23}. Despite the similarities, the absence of a topology and the groupoid structure introduced new difficulties with respect to the theory developed by Monod. For instance, in order to keep the parallelism with bounded cohomology of groups, the authors suitably modified the notion of \emph{coefficient module} \cite[Definition 1.2.1]{monod:libro} by replacing the continuity of the action with measurability \cite[Definition 3.3]{sarti:savini:23}. However, the fibred nature of groupoids and of the complexes defining their cohomology suggest that coefficients should have a compatible fibred structure. Precisely, the natural candidate to replace Banach spaces in our framework are \emph{measurable bundles of Banach spaces}. The goal of this paper is exactly to define bounded cohomology of measured groupoids with such coefficients. 

Roughly speaking, given a measured groupoid $\mathcal{G}$ over $X$, a $\mathcal{G}$-bundle of Banach spaces is the datum of a family of Banach spaces $(E_x)_{x\in X}$ organized in a measurable fashion and endowed with a $\mathcal{G}$-action by isometries (Definition \ref{definition_bundle}). 
A \emph{morphism} of bundles is a family of linear maps that preserves measurability (Definition \ref{definition_morphism_bundles}). Bundles and morphisms are the building blocks starting from which we construct our cohomological theory. 
To this end, we introduce \emph{relative injectivity} for bundles (Definition \ref{definition_relatively_injective}) and the notion of \emph{strong} resolutions (Definition \ref{def strong}).
Then we prove the Fundamental Lemma of Homological Algebra in this context (Lemma \ref{lemma_fundamental_lemma}).
As a consequence we deduce that any two strong resolutions by relatively injective bundles must share the same cohomology (Corollary \ref{corollary_unique_resolution}).

Once we have fixed all the necessary tools, we move on and we focus on the standard resolution of essentially bounded sections, namely the complex described in Example \ref{example_bundle_linf} and whose generic bundle is given by
\begin{equation*}
 \mathcal{L}(\mathcal{G}^{(\bullet+1)},\mathcal{E})\,,\;\;\;x\mapsto \Linfw((\mathcal{G}^{(\bullet+1)},\nu_x^{\bullet+1}),E_x)\,.
\end{equation*}
For $t$-discrete groupoids, the space of essentially bounded sections of the above bundle, together with the standard homogeneous coboundary operator, defines the bounded cohomology of $\mathcal{G}$. 
The reason why we restrict only to $t$-discrete groupoids is subtle and explained at the beginning of Section \ref{section:bounded:cohomology}.
\begin{intro_defn}\label{def bounded cohomology}
  The \emph{bounded cohomology} $\Hmb^{\bullet}(\mathcal{G},\mathcal{E})$ of a $t$-discrete measured groupoid $\mathcal{G}$ with coefficients in a dual measurable $\mathcal{G}$-bundle $\mathcal{E}$
 is the cohomology of the complex $(\Linf(X, \mathcal{L}(\mathcal{G}^{(\bullet+1)},\mathcal{E}))^{\mathcal{G}}, d^{\bullet})$, namely
$$\Hmb^{k}(\mathcal{G},\mathcal{E})\coloneqq \Hm^k(\Linf(X, \mathcal{L}(\mathcal{G}^{(\bullet+1)},\mathcal{E}))^{\mathcal{G}}, d^{\bullet})\,.$$
\end{intro_defn}

Given an \emph{amenable} $\mathcal{G}$-space $(S,\tau)$ and a $\mathcal{G}$-bundle $\mathcal{E}$, we consider an analogous complex with generic bundle
$$\mathcal{L}(S^{(\bullet+1)},\mathcal{E})\,,\;\;\;x\mapsto \Linfw((S^{(\bullet+1)},\tau_x^{\bullet+1}),E_x)\,,$$
and we show that the corresponding resolution satisfies the hypothesis of the Fundamental Lemma of Homological Algebra. Precisely, we prove the following

  \begin{intro_thm}\label{thm:amenable}
    Let $\calG$ be a $t$-discrete measured groupoid, $(S,\tau)$ an amenable $\mathcal{G}$-space and $\mathcal{E}$ a measurable $\calG$-bundle that is the dual of a separable measurable $\calG$-bundle.
    Then we have a natural isomorphism
    $$\Hm^{k}(\Linf(X, \mathcal{L}(S^{(\bullet+1)},\mathcal{E}))^{\mathcal{G}})\cong \Hmb^{k}(\calG,\mathcal{E})\,$$
    for every $k\geq 0$.
    \end{intro_thm}

To sum up, the bounded cohomology of a $t$-discrete measured groupoid $\calG$ can be computed via the resolutions of essentially bounded sections on any amenable $\calG$-space. A relevant example of such space is given by a $\calG$-boundary \cite{sarti:savini:24}. As a consequence any $\calG$-boundary can be used to compute the bounded cohomology of $\mathcal{G}$ (Corollary \ref{corollary:boundaries}). This extends known results proved  by Ivanov for discrete groups and by Monod in the continuous setting. 

As a consequence of Theorem \ref{thm:amenable} we get a new proof of the following result, which first appeared in \cite{sarti:savini:23}.

\begin{intro_cor}\label{cor vanishing amenable}
Let $\mathcal{G}$ a $t$-discrete amenable measured groupoid and $\mathcal{E}$ a measurable $\calG$-bundle that is the dual of a separable measurable $\calG$-bundle. Then we have that
$$
\Hmb^{k}(\calG,\mathcal{E}) \cong 0,
$$
for $k \geq 1$. 
\end{intro_cor}

    \vspace{5pt} 
    \paragraph{\textbf{Structure of the paper.}}
    The paper is divided in three sections. In Section \ref{section:preliminarie} we introduce some background material, precisely we recall the basics about groupoids (Section \ref{section:groupoids}) and about amenability (Section \ref{section:amenability}). 
    Then, in Section \ref{section measurable bundles}, we focus on measurable bundles, providing a brief overview of the theory (Section \ref{section definition examples}), then proving a disintegration isomorphism for integrable and essentially bounded sections (Section \ref{section disintegration}) and finally introducing the necessary homological tools for our theory (Section \ref{section:homological}). In Section \ref{section:bounded:cohomology}, we define bounded cohomology, we prove all the results stated in the introduction and we discuss some consequences. 
    \vspace{5pt}

    \paragraph{\textbf{Acknowledgements.}} The authors were partially supported by INdAM through GNSAGA.
    The first author's research is funded by MUR through the PRIN project ``Geometry and topology of manifolds". 

\section{Preliminaries}\label{section:preliminarie}
We start with a list of some tools and notions that we are going to exploit along the paper. We warn the reader that we will not be exhaustive and we will sometimes avoid technicalities to make the exposition clearer. For any detail we refer to the book by Anantharaman-Delaroche and Renault \cite{delaroche:renault}. Here the notation is the same adopted by the authors in our previous papers \cite{sarti:savini:23,sarti:savini:24}. 

\subsection{Measured groupoids}\label{section:groupoids}

A \emph{groupoid} is a small category whose morphisms are invertible. We denote the space of objects (the \emph{unit space}) as $X$ and the space of morphisms (the \emph{groupoid}) as $\mathcal{G}$. With $t: \mathcal{G}\to X$ and $s: \mathcal{G}\to X$ we refer to the \emph{target} map and the \emph{source} map, respectively. 
For every $x\in X$, we write $\mathcal{G}^x=t^{-1}(x)$ and 
$\mathcal{G}_x=s^{-1}(x)$.

Given such a groupoid, a \emph{left $\mathcal{G}$-space} is a set $S$ endowed with a map $t_S:S\to X$ satisfying the following conditions 
 \begin{itemize}
 \item[(i)] $(gh)s=g(hs)$, whenever $(gh,s)\in \mathcal{G}* S$ and $(g,hs) \in \mathcal{G} *S $;
 \item[(ii)] $t_S(gs)=t(g)$ whenever $(g,s) \in \mathcal{G}* S$;
 \item[(iii)] $gg^{-1}s =g^{-1} gs=s$, whenever $(g,s)\in \mathcal{G}* S$.
 \end{itemize}
Here $\mathcal{G} \ast S$ denotes the fibred product $\mathcal{G} \ast S=\{ (g,s) \in \mathcal{G} \times S \ | \ s(g)=t_S(s) \}$. For a left $\mathcal{G}$-space, following Anatharaman-Delaroche and Renault \cite[Chapter 2.a]{delaroche:renault}, we endow the fibred product $S \ast \mathcal{G}=\{ (s,g) \in S \times \mathcal{G} \ | \ t_S(s)=t(g) \}$, with the groupoid structure having units $S$, target $t(s,g)=s$, source $s(g,s)=g^{-1}s$ and inverse $(s,g)^{-1}=(g^{-1}s,g^{-1})$. We call $S \ast \mathcal{G}$ \emph{semidirect groupoid} and we denote it by $S \rtimes \mathcal{G}$. 

A \emph{Borel groupoid} is a groupoid endowed with a $\sigma$-algebra such that the composition and the inverse map are measurable. 
Since we are interested in groupoids equipped with a measure, from now on we assume that the unit space $X$ is a standard Borel space and $\mu$ is a probability measure on it. 

A \emph{Borel Haar system} of measure on $\mathcal{G}$ is a family $\rho=\{\rho^x \}_{x \in X}$ of $\sigma$-finite measures with $\rho^x(\mathcal{G}\setminus \mathcal{G}^x)=0$ for every $x$, such that the map 
$$x\mapsto \rho^x(f):=\int_{\mathcal{G}} f(g)d\rho^x(g)$$ is measurable whenever $f: \mathcal{G}\to \mathbb{R}$ is so and it holds that
\begin{equation}\label{equation:Haar}
  \int_{\calG} f(g h)d \rho^{s(g)} (h)=\int_{\calG} f(h)d \rho^{t(g)}(h)\,
\end{equation}
for every $g\in \mathcal{G}$.

The composition of a Haar system $\rho$ with the measure $\mu$ gives back a measure on $\mathcal{G}$ defined by
$$\rho \circ \mu (f)\coloneqq \int_X \rho^x (f) d\mu(x)\,.$$
If the measure class of $\rho \circ \mu$ is unchanged by the inverse map, we say that $\rho \circ \mu$ is \emph{quasi-invariant}. 

\begin{defn}
A \emph{measured groupoid} is a Borel groupoid $\mathcal{G}$ endowed with a quasi-invariant measure of the form $\rho \circ \mu$, where $\rho$ is a Haar system and $\mu$ a probability measure on $X$. 
\end{defn}

It is often useful to replace $\rho\circ \mu$ with an equivalent \emph{probability} measure. As shown by Renault \cite{renault80}, one can find such a measure $\nu\sim \rho\circ \mu$ that admits a disintegration $$\nu=\int_X\nu^x d\mu(x)\,.$$ 
In this context we will have that 
\begin{equation}\label{quasi invariant system}
g_\ast\nu^{s(g)} \sim \nu^{t(g)},
\end{equation}
namely they are not the same measure, but they share the same measure class. From now on, we will consider a measured groupoid equipped with a quasi-invariant probability measure on it.


When $\mathcal{G}^x$ is countable and $\rho^x$ is precisely the counting measure for every $x \in X$, we say that $\mathcal{G}$ is a \emph{$t$-discrete groupoid}.

Given a measured groupoid $(\mathcal{G},\nu)$, any left $\mathcal{G}$-space $S$ is assumed to be a Borel space endowed with a measure $\tau$ that disintegrates with respect to $t_S$, namely
$$
\tau=\int_X \tau^x d\mu(x).
$$
Additionally, for the family of measures $\{ \tau^x \}_{x \in X}$, we require that 
$$
g_\ast \tau^{s(g)} \sim \tau^{t(g)},
$$ 
namely it holds something similar to Equation \eqref{quasi invariant system}. With the above assumptions, $S \rtimes \mathcal{G}$ has a natural structure of measured groupoid over $S$, whose measure is given by $\nu \circ \tau$. Here the system $\{ \nu^x \}_{x \in X}$ is extended to a system parametrized by $S$ with the help of the map $t_S$, that is $\nu^s:=\nu^{t_S(s)}$.

\subsection{Amenability}\label{section:amenability}
Amenability of measured groupoids has several equivalent definitions, for instance via a fixed point property or through the notion of means. In what follows we are going to give a \emph{fibred version} of classic amenability of groups passing through the notion of invariant measurable systems of means introduced by Anantharaman-Delaroche and Renault \cite{delaroche:renault}. 

  Let $(\calG,\nu)$ be a measured groupoid and let $\mu$ be the probability measure on the unit space $X$. An \emph{invariant measurable system of means} is a family $\{\mathfrak{m}^x\,|\, x\in \calG^{(0)}\}$ of linear functions $\mathfrak{m}^x: \Linf(\mathcal{G},\nu^x)\rightarrow \mathbb{R}$ of norm one such that 
  the map $x\mapsto \mathfrak{m}^x(\lambda)$ is Borel for every $\lambda\in \Linf(\mathcal{G})$
  and for $\nu$-almost every $g\in \calG$ one has
  \begin{equation}
  g\mathfrak{m}^{s(g)}=\mathrm{m}^{t(g)},
  \end{equation}
  which means
  \begin{equation}\label{equation_invariance_mean}
    \mathfrak{m}^{t(g)} (\lambda^{t(g)})=\mathfrak{m}^{s(g)}(g^{-1}\lambda^{t(g)})\,.
  \end{equation}
In the above equation we exploited the disintegration isomorphism \cite[Equation (12)]{sarti:savini:23} which allows us to write an element $\lambda \in \mathrm{L}^\infty(\mathcal{G})$ as an essentially bounded section $\{\lambda^x\}_{x \in X}$ of a suitable measurable bundle, namely $\lambda^x \in \mathrm{L}^\infty(\mathcal{G},\nu^x)$. With this notation, we have that  
$$
(g^{-1}\lambda^{t(g)})(h)=\lambda^{t(g)}(gh),
$$
for every $h \in \mathcal{G}^{s(g)}$. We refer the reader to the next section for more details about measurable bundles and about the disintegration isomorphism. 

  \begin{defn}\label{definition_amenable_groupoid}
  A measured groupoid $\calG$ is \emph{amenable} if it admits an invariant measurable system of means.

A $\mathcal{G}$-space $(S,\tau)$ is \emph{amenable} if the measured groupoid $S \rtimes \mathcal{G}$ is amenable. 
\end{defn}
  
Given two left $\mathcal{G}$-spaces $(S,\tau)$ and $(T,\theta)$, we can consider the fibred product $S \ast T$ with respect to the maps $t_S:S \rightarrow X$ and $t_T:T \rightarrow X$. The latter is stills a left $\mathcal{G}$-space with the standard diagonal action. Additionally we can endow $S \ast T$ with a quasi-invariant probability by defining the fibred product measure
$$
\tau \ast \theta:=\int_X \tau^x \otimes \theta^x d\mu(x). 
$$
In this way $(S\ast T) \rtimes \mathcal{G}$ becomes naturally a measured groupoid. 

\begin{lemma}\label{lemma:rel:inj1}
  Let $(S,\tau)$ and $(T,\theta)$ two left $\mathcal{G}$-spaces. If $S$ is $\mathcal{G}$-amenable, then also the fibred product $(S*T,\tau\ast \theta)$ is $\mathcal{G}$-amenable.
\end{lemma}

\begin{proof}
  In order to prove the claim we need to exhibit an equivariant family of means $\{\mathfrak{m}^{(s,t)}\}_{(s,t)\in S*T}$. If we set $x=t_S(s)$, we define 
  $$\mathfrak{m}^{(s,t)}: \Linf(\mathcal{G},\nu^x)\to \mathbb{R}\,,\;\;\; \mathfrak{m}^{(s,t)}(\varphi)\coloneqq \mathfrak{m}^{s}(\varphi)\,.$$
  The measurability of the map
  $$(s,t)\mapsto \mathfrak{m}^{(s,t)}(\varphi)$$ follows by the fact that $s\mapsto\mathfrak{m}^{s}(\varphi)$ is measurable, by assumption. 
  Moreover, the assignment $(s,t)\mapsto \mathfrak{m}^{(s,t)}$ is $\mathcal{G}$-equivariant, indeed
  $$g \mathfrak{m}^{(g^{-1}s,g^{-1}t)}=g \mathfrak{m}^{g^{-1}s}= \mathfrak{m}^{s}= \mathfrak{m}^{(s,t)}\,.$$
This concludes the proof. 
\end{proof}

\section{Measurable Bundles of Banach spaces}\label{section measurable bundles}

\subsection{Definition and examples}\label{section definition examples}

In this section we introduce \emph{measurable bundles of Banach spaces}, which are the main object of our theory. Suitable references for this topic are the book by Fell and Doran \cite{fell:doran} and the one by Anantharaman-Delaroche and Renault \cite{delaroche:renault}. 

\begin{defn}\label{definition_bundle}
A \emph{measurable bundle of Banach spaces} (or simply a \emph{measurable bundle}) over a standard Borel probability space $(X,\mu)$ is a family of Banach spaces $\calE=(E_x)_{x\in X}$ endowed with a \emph{measurable structure}, that is a collection $\calM$ of vector fields $\sigma: x\mapsto \sigma(x)\in E_x$ such that 
\begin{itemize}
 \item[(1)] $\sigma_1+\sigma_2\in \calM$ whenever $\sigma_1,\sigma_2 \in \calM$;
 \item[(2)] $\varphi \cdot \sigma \in \calM$ whenever $\sigma\in \calM$ and $\varphi$ is $\mu$-measurable on $X$;
 \item[(3)] if $\sigma\in \calM$ then $x\mapsto \lVert \sigma(x) \rVert_{E_x}$ is $\mu$-measurable;
 \item[(4)] if $(\sigma_n)$ is a net in $\calM$ and $\sigma_n(x)\rightarrow \sigma(x)$ for $\mu$-almost every $x\in X$, then $\sigma \in \calM$.
 \end{itemize} 
An element $\sigma\in \calM$ is called \emph{measurable section} of $\calE$.
  
 A measurable bundle of Banach spaces is \emph{separable} if there exists a countable family $\{\sigma_n\}_{n \in \mathbb{N}}$ of sections such that $(\sigma_n(x))_{n\in \mathbb{N}}$ is dense in $E_x$, for almost every $x\in X$. 

\end{defn}

In the theory of measurable bundles it is often useful to decide whether a family of vector fields determines a measurable structure for the bundle. Given a linear space of vector fields $\calQ$ such that $x\mapsto \lVert \sigma(x) \rVert_{E_x}$ is $\mu$-measurable for every $\sigma\in \calQ$, we can construct the \emph{measurable structure generated by} $\calQ$. This is the smallest measurable structure containing $\calQ$ and it is constructed as follows: we first take the set 
$$\mathcal{R}\coloneqq\{\varphi \cdot \sigma\,|\, \varphi:X\rightarrow \mathbb{C} \; \text{simple function}\,,\, \sigma\in \mathcal{Q}\}$$
and then we consider all the vector fields that can be realized as limits almost everywhere of elements in $\mathcal{R}$.

\begin{oss} \label{remark_bundles_dlp}
Separable measurable bundles of Banach spaces are a particular instance of measurable fields of metric spaces studied by Arino, Delode and Penot \cite{Arino} and by Anderegg and Henry \cite{anderegg:henry:14}. 
%
%
Given a measurable field $E$ of metric spaces over $X$, Duchesne, Lecureux and Pozzetti \cite[Lemma 4.12]{duchesne:lecureux:pozzetti:18} proved that $X$ contains a full-measure subset $X_0$ such that the set
$$E=\bigsqcup\limits_{x\in X_0} E_x $$
admits a standard Borel structure turning the projection 
$E\rightarrow X_0$ into a Borel map. Additionally, the set of measurable sections of $X_0\rightarrow E$ corresponds to the sections in $\calM$ (see also \cite[Proposition 1.9]{Arino}). 
\end{oss}

We describe below some examples that we are going to use in the rest of the paper. In order to simplify the exposition, we will leave some technical but elementary details to the reader.

\begin{es}\label{es subbundle}
Let $\mathcal{E}$ be a measurable bundle of Banach spaces over a standard Borel probability space $(X,\mu)$. Let $\mathcal{M}$ be the measurable structure on $\mathcal{E}$. Given a family $\mathcal{K}=(K_x)_{x \in X}$, where $K_x < E_x$ is a Banach subspace of $E_x$, we can consider the restriction of $\mathcal{M}$ to $\mathcal{K}$, namely
$$
\mathcal{M}|_{\mathcal{K}}:=\{ \sigma \in \mathcal{M} \ | \ \sigma(x) \in K_x, \ \forall x \in X \}. 
$$
In this way we obtain a new measurable bundle $(\mathcal{K},\mathcal{M}|_{\mathcal{K}})$ of Banach spaces over $X$ which is called \emph{subbundle} of $(\mathcal{E},\mathcal{M})$. 
\end{es}

\begin{es}\label{example_bundle_pull_back}
Given a Borel map $\pi:Y\rightarrow X$ between standard Borel probability spaces $(Y,\nu)$ and $(X,\mu)$, a measurable bundle of Banach spaces $\calE=(E_x)_{x\in X}$ over $X$ gives rise to a measurable bundle $\mathcal{F}=\pi^*\calE$ over $Y$ called \emph{pullback} of $\calE$ via $\pi$. Precisely, we set 
$$F_y=(\pi^\ast E)_y:=E_{\pi(y)}$$ for every $y\in Y$ and we consider the measurable structure $\pi^*\calM$ generated by the set  
$$\calN\coloneqq \{ \sigma\circ \pi \,,\, \sigma\in \calM\},$$ where
$\calM$ is a measurable structure for $\calE$.

Notice that if $\mathcal{E}$ is separable then $\mathcal{F}=\pi^*\mathcal{E}$ is separable too: a countable dense family of sections for the latter can be constructed starting from the one for $\mathcal{E}$ and precomposing each section with $\pi$.
\end{es}

\begin{es}\label{example_dual_bundle}
Let $\mathcal{E}=(E_x)$ be a separable measurable bundle over $(X,\mu)$ and consider the dual bundle $\mathcal{E}^*=(E_x^*)$.  Given $\mathcal{M}$ a measurable structure for $\mathcal{E}$, we can consider the family
$$\mathcal{M}^*\coloneqq \{ x\mapsto \eta(x)\in E_x^*\,|\, x\mapsto \langle  \eta(x),\sigma(x)\rangle \text{ is } \mu-\text{measurable } \; \forall \;\sigma\in \mathcal{M}  \}\,.$$
By \cite[Lemma A.3.7]{delaroche:renault}, the separability of $\mathcal{E}$ implies that $\mathcal{M}^*$ satisfies conditions (i)-(iv) of Definition \ref{definition_bundle}. Thus $\mathcal{E}^*$ is a measurable bundle of Banach spaces, which is not separable in general.
\end{es}

The presence of a norm on each fiber would suggest to consider the analogous of \emph{$\text{L}^p$-spaces}. This can be done by taking measurable sections whose norm is $p$-integrable as a function on $(X,\mu)$. We will be particularly interested in the cases when $p=1,\infty$.
\begin{defn}
Let $\mathcal{E}$ be a measurable bundle over a standard Borel probability space $(X,\mu)$ with measurable structure $\mathcal{M}$. Given a measurable section $\sigma \in \mathcal{M}$, we define
$$
\lVert \sigma \rVert: X \rightarrow [0,\infty), \ \ \ \lVert \sigma \rVert(x):=\lVert \sigma(x) \rVert_{E_x}. 
$$
The space of \emph{integrable sections} is 
$$\Lone(X,\mathcal{E})\coloneqq \{\sigma\in \mathcal{M}\,|\, \mu(\lVert \sigma \rVert)<+\infty \}/_{\sim_{\mu}}\,,$$
where we identify sections that coincide $\mu$-almost everywhere. This space is naturally normed by
$$
\lVert [ \sigma ] \rVert_{\Lone(X,\mathcal{E})}:=\mu( \lVert \sigma \rVert ). 
$$
The space of \emph{essentially bounded sections} is given by
$$\Linf(X,\mathcal{E})\coloneqq \{\sigma\in \mathcal{M}\,|\,  \lVert \sigma \rVert \in\Linf(X) \}/_{\sim_{\mu}}\,,$$
where again we identify sections that coincide $\mu$-almost everywhere. Also in this case, we have a natural norm defined by
$$
\lVert [ \sigma ] \rVert_{\Linf(X,\mathcal{E})}:=\lVert \sigma \rVert_{\Linf(X)}.
$$
\end{defn}

With an abuse of notation, we will refer to elements of either $\Lone(X,\mathcal{E})$ or $\Linf(X,\mathcal{E})$ by dropping the parenthesis of the equivalence class and considering a generic representative. 

We conclude the section with the following duality isomorphism.

\begin{prop}\cite[Proposition A.3.9]{delaroche:renault}\label{prop duality isomorphism delaroche}
Let $\mathcal{E}$ be a separable measurable bundle of Banach spaces over a standard Borel probability space $(X,\mu)$. Then we have a canonical isomorphism
$$
\mathrm{L}^\infty(X,\mathcal{E}^\ast) \cong \mathrm{L}^1(X,\mathcal{E})^\ast
$$
\end{prop}

\subsection{The disintegration isomorphism}\label{section disintegration}
In what follows we show a disintegration isomorphism similar to the one exposed in \cite[Section 4.1]{sarti:savini:23}, but valid in the more general context of measurable bundles of Banach spaces.

We fix a Borel map $\pi:Y\rightarrow X$ between standard Borel probability spaces $(Y,\nu)$ and $(X,\mu)$ and a \emph{separable} measurable bundle of Banach spaces $\calE=(E_x)_{x\in X}$ over $X$. If we assume that $\pi_\ast \nu=\mu$, by Hahn disintegration theorem \cite[Theorem 2.1]{Hahn} we must have that
$$\nu=\int_X \nu^xd\mu(x)\,,$$
where $\nu^x$ is a probability measure on $Y$ such that $\nu^x(Y \setminus \pi^{-1}(x))=0$ and 
$$
x \mapsto \nu^x(f)
$$
is Borel for every bounded Borel function $f$ on $Y$. 

Following Example \ref{example_bundle_pull_back}, we consider $\mathcal{F}= \pi^* \mathcal{E}$ and its measurable structure $\mathcal{N}$. We can define the space of integrable sections $\mathrm{L}^1(Y,\mathcal{F})$. On the other hand, we consider the bundle
$$ \mathcal{L}(Y,\mathcal{E}): x \longrightarrow \mathrm{L}^1((Y,\nu^x),E_x),$$
endowed with the measurable structure generated by 
\begin{equation}\label{eq family L1}
\mathcal{Q}\coloneqq \{\eta\in \mathcal{N}\,|\, \nu^x( \lVert \eta \rVert)<+\infty \text{ for } \mu-\text{a.e. } x\in X \}\,,
\end{equation}
which exists by \cite[Proposition 4.2]{fell:doran}. In fact any measurable section $\sigma \in \mathcal{Q}$ defines a vector field of $ \mathcal{L}(Y,\mathcal{E})$ by setting 
\begin{equation}\label{eq eta x}
x \mapsto \sigma^x(y)=
\begin{cases*}
0 &  $\text{if } y \notin \pi^{-1}(x)$\\
\sigma(y) & \textup{otherwise}. 
\end{cases*}
\end{equation}
If we consider the space of integrable sections $\mathrm{L}^1(X,  \mathcal{L}(Y,\mathcal{E}))$, we have the following:

\begin{thm}\label{proposition_disintegration_isomorphism}
  The function 
  $$\Phi: \mathrm{L}^1(Y,\mathcal{F})\rightarrow \Lone(X,  \mathcal{L}(Y,\mathcal{E}))\,,\;\;\;
  \sigma \mapsto (x\mapsto \sigma^x)$$
is an isometric isomorphism of Banach spaces. 
\end{thm}
\begin{proof}
 We start noticing that $\Phi$ is well-defined. Indeed, given a measurable section $\sigma \in \mathcal{N}$ such that $\nu(\lVert \sigma \rVert)$ is finite, by Hahn disintegration theorem \cite[Theorem 2.1]{Hahn} we have 
\begin{equation} \label{equation phi isometry}
\nu(\lVert \sigma \rVert)=\int_Y \lVert \sigma \rVert(y) d\nu(y)=\int_X \int_Y \lVert \sigma \rVert(y) d\nu^x(y)d\mu(x)=\int_X \nu^x( \lVert \sigma \rVert)d\mu(x).
\end{equation}
As a consequence $\nu^x(\lVert \sigma \rVert)$ must be finite for $\mu$-almost every $x \in X$. The same formula shows that if $\sigma$ and $\sigma'$ differ by a set of $\nu$-measure zero, then $\Phi(\sigma)$ and $\Phi(\sigma')$ define the same class. Equation \eqref{equation phi isometry} proves also that $\Phi$ must be an isometry on the image: indeed we have that
\begin{align*}
\lVert \sigma \rVert_{\Lone(Y,\mathcal{F})}&=\int_X\int_Y \lVert \sigma \rVert(y)d\nu^x(y)d\mu(x)\\
&=\int_X \lVert \Phi(\sigma)(x) \rVert_{\Lone((Y,\nu^x),E_x)}d\mu(x)=\lVert \Phi(\sigma) \rVert_{\Lone(X, \mathcal{L}(Y,\mathcal{E}))}.
\end{align*}

We are left to show the surjectivity. By the separability of $\mathcal{\mathcal{E}}$, there exists a countable family of sections $(\sigma_n)_{n \in \mathbb{N}}$. Since $Y$ is standard Borel, there exists a countable family of Borel subsets $(B_k)_{k \in \mathbb{N}}$ which generates the $\sigma$-algebra. The family of sections
$$
\mathcal{S}:=\left\{ \sum_{k=1}^\ell a_k \chi_{B_k} \cdot (\sigma_k \circ \pi) \ | \ a_k \in \mathbb{Q}(i) \right\},
$$
is a subset of the generating family $\mathcal{Q}$ defined by Equation \eqref{eq family L1}. Since $\Phi$ is surjective on $\mathcal{S}$ and the latter is dense in $\mathrm{L}^1(X, \mathcal{L}(Y,\mathcal{E}))$, the claim is proved. 
\end{proof}

The last part of the proof of the previous theorem shows that the separability of $\mathcal{E}$ and the fact the $Y$ is standard Borel imply that $ \mathcal{L}(Y,\mathcal{E})$ is separable. By Example \ref{example_bundle_pull_back} the bundle $\mathcal{F}$ is separable too. Thanks to Proposition \ref{prop duality isomorphism delaroche}, we can dualize the isomorphism given in Theorem \ref{proposition_disintegration_isomorphism} to obtain
$$\Linf(Y,\mathcal{F}^*)\cong \Linf(X, \mathcal{L}(Y,\mathcal{E})^*),$$ where
$$ \mathcal{L}(Y,\mathcal{E})^*: x \longrightarrow \Linfw ((Y,\nu^x),E_x^*)\,,$$
and the measurable structure is dual to the one of $ \mathcal{L}(Y,\mathcal{E})$.


In our constructions, we will be interested into bundles of Banach spaces endowed with an isometric action by a measured groupoid. From now on, $\mathcal{G}$ will be a measured groupoid over 
a standard Borel probability space $(X,\mu)$ endowed with a probability measure $\nu=\int_X\nu^xd\mu(x)$ as described in Section \ref{section:groupoids}.

\begin{defn}\label{definition_action_on_bundle}
Let $\calG$ be a measured groupoid. Given a measurable bundle $\mathcal{E}$ of Banach spaces over $X$, we define the fibred product
$$
\calG \ast \mathcal{E} := \{ (g,v) \ |  v \in E_{s(g)} \}.
$$
An \emph{isometric left $\calG$-action} (or simply a \emph{$\mathcal{G}$-action}) on $\mathcal{E}$ is a map 
$$
L:\calG \ast \mathcal{E} \longrightarrow \mathcal{E}, \ \ (g,v) \mapsto L(g)v
$$
such that for, any $g\in \mathcal{G}$, the function 
$$L(g): E_{s(g)}\rightarrow E_{t(g)}, \ \ v \mapsto L(g)v$$
is a linear isometry and, if $\sigma$ is a measurable section for $\mathcal{E}$ on $X$, then the map $g\mapsto L(g) \sigma(s(g))$
is a measurable section of $t^*\calE$ on $\calG$.

A \emph{measurable $\calG$-bundle $(\mathcal{E},L)$ of Banach spaces} is a measurable bundle $\mathcal{E}$ endowed with an isometric left $\calG$-action $L$. 
\end{defn}


\begin{defn}\label{def invariant section}
Let $\calG$ be a measured groupoid. Given a measurable $\calG$-bundle of Banach spaces $(\mathcal{E},L)$ over $X$, we say that a measurable section $\sigma$ of $\mathcal{E}$ is \emph{almost $\calG$-invariant} if we have
$$
L(g)\sigma(s(g))=\sigma(t(g)),
$$
for almost every $g \in \calG$. We will denote by $\Lone(X,\mathcal{E})^\calG$ (respectively $\Linf(X,\mathcal{E})^{\calG}$) the space of integrable (respectively essentially bounded) $\calG$-invariant sections of $\mathcal{E}$. 
\end{defn}

\begin{es}\label{example_dual_bundle_action}
  Let $\calE=(E_x)_{x\in X}$ be a separable measurable bundle of Banach spaces over $X$ and denote by $\calM$ its measurable structure. 
  Consider the dual bundle $\calE^*=(E_x^*)$.
  If $\calE$ is endowed with an isometric left action $L:\calG *\mathcal{E}\rightarrow \mathcal{E}$ of a measured groupoid $\calG$, then we can define the \emph{dual action} $L^\ast:\calG \ast \mathcal{E}^\ast \rightarrow \mathcal{E}^\ast$ in the following way
  $$\langle L^\ast(g)\theta ,v\rangle \coloneqq \langle \theta,  L(g)^{-1}v\rangle $$
  for every $\theta \in E^\ast_{s(g)}$ and $v \in E_{t(g)}$. 
  \end{es}

The next example will be the main ingredient in the definition of the bounded cohomology of a groupoid with bundle coefficients. 

\begin{es}\label{example_bundle_linf}
  Let $\calG$ be a measured groupoid. Consider a measurable $\calG$-bundle $(\mathcal{E},L)$ of Banach spaces over $X$ which is the dual of a separable measurable $\mathcal{G}$-bundle $(\mathcal{E}^{\flat},L^\flat)$. This means in particular that the predual action is fixed once and for all. 

Since we can disintegrate $$\nu=\int_X\nu^xd\mu(x),$$ Theorem \ref{proposition_disintegration_isomorphism} guarantees the
isometric isomorphism
\begin{equation}\label{eq L1 for G}
\Lone(\mathcal{G},\mathcal{F})\cong \Lone(X, \mathcal{L}^{\flat}(\mathcal{G},\mathcal{E})),
\end{equation}
where $\mathcal{F}=t^\ast \mathcal{E}^\flat$ and $ \mathcal{L}^{\flat}(\mathcal{G},\mathcal{E}):x \longrightarrow \Lone((\mathcal{G},\nu^x),E^{\flat}_x)$ is endowed with the measurable structure generated by Equation \eqref{eq family L1}.

By Proposition \ref{prop duality isomorphism delaroche}, we can dualize Equation \eqref{eq L1 for G} to obtain 
$$
  \Linf(\calG,\mathcal{F}^\ast)\cong \Linf(X, \mathcal{L}(\mathcal{G},\mathcal{E}))\,
$$
  where $ \mathcal{L}(\mathcal{G},\mathcal{E}):x \longrightarrow \Linfw((\mathcal{G},\nu^x),E_x)$ and the measurable structure is dual to the one by Equation \eqref{eq family L1}. Moreover, the bundle $ \mathcal{L}(\mathcal{G},\mathcal{E})$ admits a natural isometric left $\calG$-action given by
$$
\overline{L}:\calG \ast  \mathcal{L}(\mathcal{G},\mathcal{E}) \rightarrow  \mathcal{L}(\mathcal{G},\mathcal{E}), \ \ (g,\lambda) \mapsto \overline{L}(g)\lambda,
$$
where $(\overline{L}(g)\lambda)(g_0)=L(g)\lambda(g^{-1}g_0)$ and $L$ denotes the isometric left $\mathcal{G}$-action on $\mathcal{E}$. 

The above argument can actually be generalized to the fibred target map $t^{(\bullet+1)}: \mathcal{G}^{(\bullet+1)} \to X$, so that we obtain isometric isomorphisms
\begin{equation}\label{equation_disintegration_linf}
  \Linf(\mathcal{G}^{(\bullet+1)},(\mathcal{F}^{\bullet+1})^\ast)\cong \Linf(X, \mathcal{L}(\mathcal{G}^{(\bullet+1)},\mathcal{E}))\,.
\end{equation}
where $\mathcal{F}^{\bullet+1}=(t^{\bullet+1})^\ast \mathcal{E}^\flat$ and $ \mathcal{L}(\mathcal{G}^{(\bullet+1)},\mathcal{E}):x \longrightarrow \Linfw((\mathcal{G}^{(\bullet+1)},\nu_x^{(\bullet+1)}),E_x)$. 
Moreover, the isometric left $\mathcal{G}$-action 
$$
\overline{L}^{(\bullet+1)}:\calG \ast  \mathcal{L}(\mathcal{G}^{(\bullet+1)},\mathcal{E}) \rightarrow  \mathcal{L}(\mathcal{G}^{(\bullet+1)},\mathcal{E}), \ \ (g,\lambda) \mapsto \overline{L}^{(\bullet+1)}(g)\lambda,
$$
is similarly defined by $\overline{L}(g)\lambda(g_0,\ldots,g_\bullet)=L(g)\lambda(g^{-1}g_0,\ldots,g^{-1}g_\bullet)$. 
More generally, the above construction and the isomorphism of Equation \eqref{equation_disintegration_linf} can be generalized to any $\mathcal{G}$-space $S$. Such bundle will be the main object of Section \ref{section:bounded:cohomology}.

We finally point out that, when $\mathcal{E}$ is the constant bundle, Equation \eqref{equation_disintegration_linf} boils down to \cite[Equation (10)]{sarti:savini:23}.

  \end{es}


\subsection{Homological algebra for measurable bundles}\label{section:homological}

In this section we will define the main framework that we need to introduce our cohomological theory. We start with the following:

\begin{defn}\label{definition_morphism_bundles}
Given two measurable bundles $\calE$ and $\calF$ of Banach spaces over $(X,\mu)$, a \emph{morphism} between them is a collection $\varphi=(\varphi_x)_{x\in X},$ where each $\varphi_x:E_x \rightarrow F_x$ is a bounded linear map such that
\begin{itemize}
  \item[(i)] for every measurable section $\sigma$ of $\mathcal{E}$
  the evaluation $x\mapsto \varphi_x(\sigma(x)) $ is a measurable section of $\mathcal{F}$;
  \item[(ii)] it is \emph{bounded}, namely there exists a constant $C$ such that
  $$\lVert \varphi_x \rVert<C,$$
  for almost every $x \in X$. We denoted by $$\lVert \varphi_x \rVert:=\sup\limits_{\lVert v \rVert_{E_x}=1} \lVert\phi_x(v)\rVert_{F_x}$$ the usual operator norm. 
\end{itemize}

If $(\calE,L_{\mathcal{E}})$ and $(\calF,L_{\mathcal{F}})$ are endowed with isometric left $\calG$-actions, a \emph{$\mathcal{G}$-morphism} is a morphism such that 
$$\varphi_{t(g)}(L_{\mathcal{E}}(g) v)= L_{\mathcal{F}}(g) \varphi_{s(g)}(v)\,$$
for almost every $g\in \mathcal{G}$ and every $v\in E_{s(g)}$. 
\end{defn}

\begin{oss}\label{oss morphism dense subset}
It is worth noticing that Definition \ref{definition_morphism_bundles}(i) can be actually checked on a generating family of sections. More precisely, let $\mathcal{E}, \mathcal{F}$ be measurable bundles of Banach spaces over $(X,\mu)$. Let $\mathcal{M}$ be the measurable structure on $\mathcal{E}$ and consider a generating linear subset $\mathcal{N} \subset \mathcal{M}$. First of all, notice that we can always substitute $\mathcal{N}$ with the set
$$
\mathcal{R}:=\{ \psi \cdot \sigma \ | \ \textup{$\psi$ is simple}, \sigma \in \mathcal{N}\}.
$$
The latter is still a linear generating subspace of $\mathcal{M}$. Given a family of continuous linear operators $\varphi_x:E_x \rightarrow F_x$, suppose that $\varphi_x(\sigma(x))$ is a measurable section of $\mathcal{F}$ for every $\sigma \in \mathcal{N}$. By the linearity of each $\varphi_x$, the same statement holds also for any measurable section of $\mathcal{R}$. The latter set is dense in $\mathcal{M}$ by \cite[Lemma 4.3]{fell:doran}. By the continuity of each $\varphi_x$, we must have that $\varphi_x(\sigma(x))$ is a measurable section of $\mathcal{F}$ for every $\sigma \in \mathcal{M}$. 
\end{oss}

\begin{notation} We fix, once and for all, the following conventions. 
  \begin{itemize}
    \item[(i)] Unless otherwise specified, all bundles are assumed to have the same base space, which is implicitly denoted by $(X,\mu)$. 
    \item[(ii)] Let $\mathcal{E},\mathcal{F},\mathcal{K}$ be measurable bundles over $(X,\mu)$. Given a morphism $\varphi:\calE \rightarrow \calF$ and another morphism $\psi:\calF \rightarrow \mathcal{K}$, we denote the composition $\psi \circ \varphi:\mathcal{E} \rightarrow \mathcal{K}$ as the morphism defined by $\psi_x \circ \varphi_x$ for every $x \in X$.
    \item[(iii)] We will always consider functions on $X$ \emph{up to null sets}. In the same flavour, two morphisms $\varphi,\theta:\mathcal{E} \rightarrow \mathcal{F}$ such that
    $
    \varphi_x=\theta_x
    $
    holds for almost every $x \in X$ will be identified. Thus, all the statements about bundles morphisms (such as identities, commutativity, etc.) must be intended to hold \emph{almost everywhere} with respect to the measure on the base space. However, to lighten the notation, we will always omit this fact. 
  \end{itemize}

  \end{notation}

\begin{defn}
  A morphism $\varphi:\mathcal{E}\to\mathcal{F}$ is \emph{injective} (respectively \emph{surjective}, \emph{bijective}, \emph{isometric}) if $\varphi_x$ is injective (respectively surjective, bijective, isometric) for every $x\in X$.
\end{defn}

\begin{oss}
Given a morphism $\varphi:\mathcal{E} \rightarrow \mathcal{F}$, we can naturally define the subbundle $\mathrm{Ker}(\varphi)$ of $\mathcal{E}$ by setting
$$
\mathrm{Ker}(\varphi):x \longrightarrow \mathrm{Ker}(\varphi_x),
$$
where we consider the restricted measurable structure given by Example \ref{es subbundle}.
\end{oss}

\begin{oss}
  By Definition \ref{definition_morphism_bundles}(ii), given a morphism $\varphi: \mathcal{E}\to \mathcal{F}$ there exists a well-defined map on the space of essentially bounded measurable sections, namely
  $$\varphi^*: \Linf(X, \mathcal{E})\to \Linf(X,\mathcal{F}), \ \ \varphi^\ast(\lambda)(x):=\varphi_x(\lambda(x)).$$ 
  Assume now that $(\mathcal{E},L_{\mathcal{E}})$ and $(\mathcal{F},L_{\mathcal{F}})$ are measurable $\calG$-bundles and $\varphi$ is a $\calG$-morphism. If $\sigma$ is an almost $\calG$-invariant section of $\mathcal{E}$, one has 
  \begin{align*}
   L_{\mathcal{F}}(g)\varphi_{s(g)}(\sigma(s(g)))&= \varphi_{t(g)}(L_{\mathcal{E}}(g)\sigma(s(g)))\\
    &= \varphi_{t(g)}(\sigma(t(g)))\,.
  \end{align*} 
  Thus, $\varphi^*$ preserves almost $\calG$-invariant sections and it restricts to a map
  $$\varphi^*: \Linf(X, \mathcal{E})^{\mathcal{G}}\to \Linf(X,\mathcal{F})^{\mathcal{G}}\,.$$
\end{oss}

\begin{defn}\label{definition_admissible}
Let $\calE,\calF$ measurable $\calG$-bundles.
A morphism $\varphi:\calE\to\calF$ is \emph{admissible} if there exists a morphism
$\eta:\calF\to\calE$ such that 
\begin{itemize}
  \item[(i)] $\lVert \eta \rVert_\infty:=\mathrm{ess \ sup}\lVert \eta_x \rVert \leq 1$;
  \item[(ii)] $\varphi \circ \eta \circ \varphi =\varphi.$
\end{itemize}
\end{defn}

We observe that, when an admissible morphism $\varphi$ is also injective, the condition (ii) above means that $\eta_x$ is a left inverse for $\varphi_x$.

Admissible morphisms are the main ingredient in the definition of relatively injective bundles. 
\begin{defn}\label{definition_relatively_injective}
  A measurable $\calG$-bundle of Banach spaces $\calE$ is \emph{relatively injective} if for every 
$\calF,\calK$ measurable $\calG$-bundles, for every injective admissible $\calG$-morphism 
$\alpha:\calF\to \calK$ with left inverse $\eta:\calK \rightarrow \calF$ and for every 
$\calG$-morphism 
$\beta:\calF\to\calE$, there exists a $\calG$-morphism 
$\psi:\calK\to\calE$ with $\lVert \psi\rVert_{\infty}\leq \lVert \beta \rVert_{\infty}$ such that
$$\psi \circ \alpha =\beta,$$ 
namely the following diagram of $\mathcal{G}$-morphisms commutes
  \begin{equation}\label{diagram_relative_injectivity}
    \begin{tikzcd}
      \mathcal{F}\arrow[hook]{rr}{\alpha} \arrow{rd}[swap]{\beta}    &&\mathcal{K}\arrow[dotted]{ld}{\psi} \arrow[bend right=30]{ll}[swap]{\eta} \\
      &\mathcal{E}\,.&
    \end{tikzcd}
  \end{equation}
\end{defn}


A crucial property of admissible morphisms is that the domain can be actually decomposed as the direct sum of the kernel and its complement. 
\begin{defn}
  Let $\mathcal{E}=(E_x)_{x \in X}$ be a measurable bundle of Banach spaces. A subbundle $\mathcal{H}=(H_x)_{x \in X}<\mathcal{E}$ is \emph{complemented} if there exists another subbundle $\mathcal{K}=(K_x)_{x \in X}$ such that 
  $$H_x\oplus K_x=E_x$$
  for every $x\in X$.
\end{defn}

Following the case of Banach spaces \cite[Section 4.2]{monod:libro}, one can see that the existence of an idempotent morphism $p:\mathcal{E} \rightarrow \mathcal{E}$ such that $\mathrm{Im}(p)=\mathcal{H}$ implies that $\mathcal{H}$ is complemented. The fact that $p$ is idempotent means that $p^2=p$ and $\mathrm{Im}(p)=\mathcal{H}$ implies that $\mathrm{Im}(p)$ is actually a subbundle of $\mathcal{E}$. We call $p$ a \emph{projection} morphism on $\mathcal{H}$. 

 Let $\mathcal{E}$ be a bundle over $X$ with measurable structure $\mathcal{M}$ and consider a subbundle $\mathcal{H}<\mathcal{E}$. Suppose that there exists a projection $p: \mathcal{E}\rightarrow \mathcal{H}$. By the previous observation $\mathcal{H}$ is actually complemented. Exploiting the projection we can define a measurable structure on the quotient bundle
  $$\mathcal{Q}:x \longrightarrow Q_x\coloneqq E_x/H_x$$
as follows. First we endow each quotient $Q_x$ with the Banach structure induced by the norm
  $$\lVert [v]\rVert_{Q_x}\coloneqq \lVert v- p_x(v)\rVert_{E_x}\,.$$
If we denote by $\pi_x:E_x \rightarrow Q_x$ the quotient projection, a measurable structure for $\mathcal{Q}$ is generated by the family
  \begin{equation}\label{eq quotient bundle}
\mathcal{N}\coloneqq \{ \eta:x\mapsto \eta(x) \in \mathcal{Q}_x \,|\, \eta(x)=\pi_x(\sigma(x))\,,\, \sigma\in \mathcal{M}\}\,.
\end{equation}
Notice that the fact that $p$ is a morphism of measurable bundles guarantees that $\mathcal{N}$ satisfies the hypothesis of \cite[Proposition 4.2]{fell:doran}. With such measurable structure the projection $\pi:\mathcal{E} \rightarrow \mathcal{Q}$ becomes a morphism of measurable bundles. 

\begin{defn}
Let $\mathcal{E}$ be a measurable bundle and let $\mathcal{H}<\mathcal{E}$ be a subbundle. Suppose to have a projection $p:\mathcal{E} \rightarrow \calE$ on $\mathcal{H}$. We call \emph{quotient bundle} the pair $(\mathcal{Q},\mathcal{M})$, where $Q_x=E_x/H_x$ is the quotient Banach space and $\mathcal{M}$ is the measurable structure generated by the set of vector fields given by Equation \eqref{eq quotient bundle}. The morphism $\pi:\mathcal{E} \rightarrow \mathcal{Q}$ is called \emph{quotient map}. 
\end{defn}

The next step in our investigation is to prove a factorization property for quotient bundles. 

\begin{prop}\label{prop factorization quotient}
Let $\mathcal{E}$ be a bundle with measurable structure $\mathcal{M}$ and consider a subbundle $\mathcal{H}<\mathcal{E}$. Suppose that there exists a projection $p: \mathcal{E}\rightarrow \mathcal{H}$. If $\varphi:\mathcal{E} \rightarrow \mathcal{F}$ is a morphism of measurable bundles with $\mathcal{H} < \Ker(\varphi)$ and $\pi:\mathcal{E} \rightarrow Q$ is the quotient map, then there exists a unique morphism $\overline{\varphi}:\mathcal{Q} \rightarrow \calF$ such that
$$
\varphi=\overline{\varphi} \circ \pi. 
$$
\end{prop}

\begin{proof}
The condition $H_x < \Ker(\varphi_x)$ implies the existence of a unique map $\overline{\varphi}_x:E_x/H_x \rightarrow F_x$ such that $\overline{\varphi}_x \circ \pi_x = \varphi_x$ for every $x \in X$. It is also clear that the boundedness of the family $(\varphi_x)_{x \in X}$ implies that 
$$
\lVert \overline{\varphi}_x \rVert < \infty. 
$$
We only need to show point $(i)$ in Definition \ref{definition_morphism_bundles}. As already discussed, in Remark \ref{oss morphism dense subset} we can check that property on the family
$$\mathcal{N}\coloneqq \{ \eta:x\mapsto \eta(x) \in \mathcal{Q}_x \,|\, \eta(x)=\pi_x(\sigma(x))\,,\, \sigma\in \mathcal{M}\}\,.$$
Let $\overline{\sigma}$ be a section in $\mathcal{N}$. This means that there exists a section $\sigma$ of $\mathcal{E}$ such that $\overline{\sigma}(x)=\pi_x(\sigma(x))$ for almost every $x \in X$. As a consequence, we must have that 
$$
\overline{\varphi}_x(\overline{\sigma}(x))=\overline{\varphi}_x(\pi_x(\sigma(x)))=\varphi_x(\sigma(x)). 
$$
Since $\varphi$ is a morphism, this concludes the proof. 
\end{proof}

We are now ready to show that both the kernel and the image of an admissible morphism are complemented subbundles. 

\begin{prop}\label{proposition_complemented}
Let $\varphi:\mathcal{E}\to \mathcal{F}$ be a morphism of measurable bundles. If $\varphi$ is admissible, then both $\Ker (\varphi)$ and $\Ima (\varphi)$ are complemented. In particular, $\Ima (\varphi)$ is a subbundle of $\mathcal{F}$.
\end{prop}
\begin{proof}
We follow the line of \cite[Proposition 4.2.1]{monod:libro}. By hypothesis there exists a morphism $\eta:\mathcal{F}\to \mathcal{E}$ such that $\varphi \circ \eta \circ \varphi=\varphi$ for almost every $x \in X$. The morphism $\Id-\eta \varphi: \mathcal{E}\to \mathcal{E}$ is idempotent, namely we have that
\begin{align*}
  (\Id-\sigma \varphi)^2=\Id-\sigma\varphi\,.
\end{align*}
Additionally, its kernel is $\Ima(\eta)$ and its image is precisely $\mathrm{Ker}(\varphi)$. Thus we obtain that 
$$\Ker(\varphi)\oplus \Ima(\eta)=\mathcal{E}\,.$$
Similarly, the morphism 
$\Id- \varphi\eta: \mathcal{F}\to \mathcal{F}$ is idempotent, its image is $\Ker(\eta)$ and its kernel is $\Ima(\varphi)$. As a consequence we obtain that  
$$\Ima(\varphi)\oplus \Ker(\eta)=\mathcal{F}\,,$$
and the conclusion follows.
\end{proof}

Relative injectivity ensures that the extension problem of Diagram \eqref{diagram_relative_injectivity} has a solution whenever $\alpha$ is \emph{injective}. In order to prove the fundamental lemma of homological algebra in our setting we need to drop the injectivity assumption. Precisely, we will consider the following \emph{generalized extension problem}
\begin{equation}\label{diagram_generalized_extension_problem}
  \begin{tikzcd}
    \Ker(\alpha)\arrow[hook]{rr}{} & &  \mathcal{F} \arrow{rr}{\alpha}\arrow[swap]{rd}{\beta} & & \mathcal{K}\arrow[dotted]{ld}{?} \arrow[bend right=30]{ll}[swap]{\eta}\\
&& &\mathcal{E}\,,& 
  \end{tikzcd}    
\end{equation}
where $\alpha \circ \eta \circ \alpha=\alpha$. 

\begin{prop}\label{proposition_generalized_extension_problem}
 Let $\alpha:\calF\to \calK$ be an admissible
$\calG$-morphism and 
$\beta:\calF\to\calE$ a $\calG$-morphism. If $\mathcal{E}$ is relatively injective and $\mathrm{Ker}(\alpha)$ is a subbundle of $\mathrm{Ker}(\beta)$, then there exists a $\calG$-morphism 
$\psi:\calK\to\calE$ with $\lVert \psi \rVert_{\infty}\leq \lVert \beta \rVert_{\infty}$ such that
$$\psi \circ \alpha=\beta.$$
\end{prop}
\begin{proof}
Since $\alpha$ is admissible, by Proposition \ref{proposition_complemented} the subbundle $\Ker(\alpha)$ is complemented in $\mathcal{F}$. We can consider the quotient bundle $\mathcal{Q}\coloneqq \mathcal{F}/\Ker(\alpha)$.
 By Proposition \ref{prop factorization quotient} we have an induced $\mathcal{G}$-morphism $\overline{\alpha}: \mathcal{Q}\to \mathcal{K}$ such that the following diagram commutes
 \begin{center}
  \begin{tikzcd}
    \mathcal{F}\arrow{rr}{\alpha} \arrow{rd}[swap]{p}    &&\mathcal{K}  \\
    &\mathcal{Q}\arrow{ru}{\overline{\alpha}}\,.&
  \end{tikzcd}
\end{center}
We can collect all those maps in the following diagram
\begin{center}
  \begin{tikzcd}
    \calF\arrow{ddrr}[swap]{p}\arrow{rrrrdd}{\alpha}\arrow[bend right=30]{rrdddd}[swap]{\beta}&&&& \\
    &&&&\\
   && \mathcal{Q}\arrow{rr}[swap]{\overline{\alpha}}\arrow[swap]{dd}{\overline{\beta}} & & \mathcal{K}\arrow[dotted]{ddll}{?}\\
   &&&&\\
 &&\mathcal{E}.&& \,
  \end{tikzcd}    
\end{center}
Here $\overline{\beta}:\mathcal{Q}\to \mathcal{E}$ is the $\mathcal{G}$-morphism whose existence is guaranteed by the condition $\Ker(\alpha)<\Ker(\beta)$. 
The conclusion follows by taking the solution of the extension problem of the lower triangle, ensured by the relative injectivity of $\mathcal{E}$.
\end{proof}

We conclude the part about relative injectivity with the following useful property.

\begin{lemma}\label{lemma:rel:inj2}
  Let $\mathcal{E},\mathcal{F}$ be measurable $\mathcal{G}$-bundles and $\varphi: \mathcal{E}\to \mathcal{F}$ a $\mathcal{G}$-morphism with $\lVert \varphi \rVert_\infty \leq 1$. Assume that there exists a morphism $\sigma: \mathcal{F}\to \mathcal{E}$ such that
\begin{itemize}
  \item[(i)] $\sigma\circ \varphi=\id$;
  \item[(ii)] $\lVert \sigma \rVert_{\infty}\leq 1$.
\end{itemize}
If $\mathcal{F}$ is relatively injective, then so is $\mathcal{E}$. 
\end{lemma}
\begin{proof}
  We consider the extension problem 
  
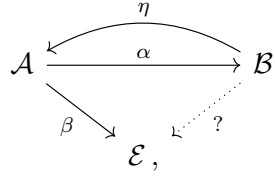
\begin{center}
  \begin{tikzcd}
    \mathcal{A} \arrow{rr}{\alpha}\arrow[swap]{rd}{\beta} & & \mathcal{B}\arrow[dotted]{ld}{?} \arrow[bend right=30]{ll}[swap]{\eta}\\
 &\mathcal{E}\,,& 
  \end{tikzcd}    
\end{center}
  Since $\mathcal{F}$ is relatively injective, there exists $\psi: \mathcal{B}\to \mathcal{F}$ of norm at most $\lVert \varphi\circ \beta \rVert_{\infty}$ such that $\varphi\circ \beta= \psi\circ \alpha$. 
  We set $\chi\coloneqq \sigma\circ \psi$, so that the following diagram commutes 
  
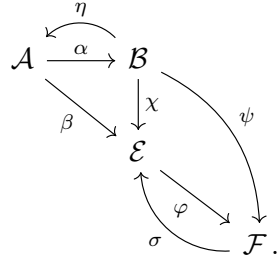
\begin{center}
    \begin{tikzcd}
      \mathcal{A}\arrow{r}{\alpha}\arrow[swap]{rd}{\beta} &   \mathcal{B} \arrow[bend left=30]{rdd}{\psi}\arrow[ bend right=45,swap]{l}{\eta}\arrow{d}{\chi}\\
    &    \mathcal{E}\arrow[swap]{rd}{\varphi} & \\
  & & \mathcal{F} \arrow[bend left=45]{lu}{\sigma}\,.
    \end{tikzcd}    
  \end{center}
  We compute 
  $$\lVert \chi \rVert_{\infty}\leq \lVert \psi \rVert_{\infty}\;\lVert \sigma \rVert_{\infty}\leq \lVert \psi \rVert_{\infty}\leq \lVert \varphi \rVert_{\infty}\lVert \beta \rVert_{\infty}\leq \lVert \beta \rVert_{\infty}\,.$$ 
Hence $\chi$ is a solution of the initial extension problem, thus $\mathcal{E}$ is relatively injective. 
\end{proof}

We move on the construction of our theory giving a list of definitions of classic notions of homological algebra translated in the framework of bundles. Since we will consider only $\mathcal{G}$-bundles and $\mathcal{G}$-morphisms, we introduce the category of complexes and their morphisms only in the equivariant setting.
\begin{defn}\label{definition_complex}
  A \emph{complex} $(\calE^{\bullet},\delta^{\bullet})$ of measurable $\calG$-bundles is a sequence
  \begin{center}
    \begin{tikzcd}
      \cdots \arrow{r}& \mathcal{E}^{\bullet-1}\arrow{r}{\delta^{\bullet-1}}&\mathcal{E}^{\bullet}\arrow{r}{\delta^{\bullet}}&\mathcal{E}^{\bullet+1}\arrow{r}&\cdots
    \end{tikzcd}
  \end{center}
  where each $\delta^{\bullet}$ is a $\calG$-morphism and $\delta^{\bullet+1}\circ \delta^{\bullet}=0$.

  A \emph{morphism} of complexes $(\calE^{\bullet},\delta^{\bullet})$ and $(\calF^{\bullet},\partial^{\bullet})$ is a sequence of $\mathcal{G}$-morphisms $\varphi^{\bullet}:\mathcal{E}^{\bullet}  \to \mathcal{F}^{\bullet}$ that commute with the coboundary operators, that is
$$
\varphi^{\bullet+1}\circ \delta^\bullet  = \partial^\bullet \circ \varphi^\bullet\,,
$$
in any degree.
\end{defn}

\begin{defn}
  A \emph{chain homotopy} between two $\mathcal{G}$-morphisms $\varphi^{\bullet},\psi^{\bullet}:(\calE^{\bullet},\delta^{\bullet})\to(\calF^{\bullet},\partial^{\bullet})$ is a family of $\mathcal{G}$-morphisms $k^{\bullet}:\mathcal{E}^{\bullet}  \to \mathcal{F}^{\bullet-1}$ such that 
  \begin{equation}\label{equation:chain:homotopy}
    \partial^{\bullet-1}\circ k^{\bullet}+k^{\bullet+1}\circ \delta^{\bullet}=\varphi^{\bullet}-\psi^{\bullet}\,,
  \end{equation}
in any degree.
\end{defn}

\begin{defn}\label{def strong}
  A \emph{contracting homotopy} for a complex $(\mathcal{E}^{\bullet},\delta^{\bullet})$ is a family of (not necessarily $\mathcal{G}$-equivariant) morphisms 
  $$k^{\bullet}:\calE^{\bullet}\to\calE^{\bullet-1}$$
   such that 
  $$\delta^{\bullet-1}\circ k^{\bullet}+k^{\bullet+1}\circ \delta^{\bullet}=\id_{\mathcal{E}^{\bullet}}$$
  and $\lVert k^\bullet \rVert \leq 1$, in any degree.

A complex $(\calE^{\bullet},\delta^{\bullet})$ of measurable $\mathcal{G}$-bundles is \emph{strong} if it admits a contracting homotopy.
\end{defn}
  
\begin{defn}
  A \emph{resolution} of a measurable $\calG$-bundle $\mathcal{E}$ is an augmented complex
  \begin{center}
    \begin{tikzcd}
      0 \arrow{r}& \mathcal{E}\arrow{r}{\epsilon}&\mathcal{E}^0\arrow{r}{\delta^{0}}&\mathcal{E}^1\arrow{r}&\cdots
    \end{tikzcd}
  \end{center}
  of measurable $\calG$-bundles $(\calE^{\bullet},\delta^{\bullet})$ with $\mathcal{E}^{\bullet}=0$ for $\bullet<0$ which is \emph{exact}, namely such that $\Ima(\epsilon)=\Ker(\delta^0)$ and $\Ima(\delta^{\bullet-1})=\Ker(\delta^\bullet)$, in any degree.
\end{defn}

\begin{oss}\label{remark_induced_complex}
  Given a complex of measurable $\calG$-bundles $(\calE^{\bullet},\delta^{\bullet})$, one can construct the complex of essentially bounded sections, namely the complex of Banach spaces
$$
\begin{tikzcd}
      \cdots \arrow{r}& \Linf(X,\mathcal{E}^{\bullet-1})\arrow{r}{d^{\bullet-1}}&\Linf(X,\mathcal{E}^{\bullet})\arrow{r}{d^{\bullet}}&\Linf(X,\mathcal{E}^{\bullet+1})\arrow{r}&\cdots
\end{tikzcd}
$$    
where $d^\bullet$ is the map induced by the coboundary $\delta^\bullet$. Since each morphism $\delta^\bullet$ is actually a $\calG$-morphism, we can restrict to the subcomplex of almost $\calG$-invariant vectors
$$
\begin{tikzcd}
      \cdots \arrow{r}& \Linf(X,\mathcal{E}^{\bullet-1})^{\calG} \arrow{r}{d^{\bullet-1}}&\Linf(X,\mathcal{E}^{\bullet})^{\calG} \arrow{r}{d^{\bullet}}&\Linf(X,\mathcal{E}^{\bullet+1})^{\calG} \arrow{r}&\cdots
\end{tikzcd}
$$ 

In the same spirit, a contracting homotopy for $(\calE^{\bullet},\delta^{\bullet})$ descends to a contracting homotopy for $(\Linf(X, \mathcal{E}^{\bullet}),d^{\bullet})$ (but not necessarily of $(\Linf(X, \mathcal{E}^{\bullet})^{\mathcal{G}},d^{\bullet})$).
\end{oss}

The following result is the analogue of the fundamental lemma of homological algebra. We refer to \cite[Section 7.2]{monod:libro} in the case of continuous bounded cohomology of groups.

\begin{lemma}\label{lemma_fundamental_lemma}
  Let $\calG$ be a measured groupoid and $\calE,\calF$ be measurable $\calG$-bundles.
  Let $(\calE^{\bullet},\delta^{\bullet})$ be a strong resolution of $\mathcal{E}$ with augmentation $\epsilon:\calE\to\calE^0$. 
  Let $(\calF^{\bullet},\partial^{\bullet})$ be a resolution of $\mathcal{F}$ by relatively injective bundles with augmentation $\varepsilon:\calF\to\calF^0$. 
  If $\varphi:\calE\to\calF$ is a $\calG$-morphism, then there exists a unique, up to chain homotopy, extension of $\varphi$ to a morphism $\varphi^{\bullet}:\mathcal{E}^{\bullet}\to\calF^{\bullet}$.
  \end{lemma}
  \begin{proof}
  We start by constructing an extension $\varphi^{0}$ of $\varphi$, and then we proceed by induction on the degree.
  Since $(\calE^{\bullet},\delta^{\bullet})$ is strong, it admits a contracting homotopy,
  say $k^{\bullet}:\mathcal{E}^{\bullet}\to\mathcal{E}^{\bullet-1}$.
  For $\bullet=0$ we have the following extension problem
  \begin{center}
    \begin{tikzcd}
  \calE \arrow{rr}{\varepsilon}\arrow[swap]{rd}{\partial \circ\varphi} && \mathcal{E}^0\arrow[bend right =30,swap]{ll}{k^0}\arrow[dotted]{dl}{?}\\
  & \mathcal{F}^0\,.&
    \end{tikzcd}    
  \end{center}
  By relative injectivity of $\mathcal{F}^0$ there exists a map 
  $$\varphi^0 :\mathcal{E}^0\to\mathcal{F}^0$$
  such that 
  $\varphi^0\circ \varepsilon= \partial\circ \varphi $.
  
  Assume now that $\varphi^{\bullet-1}$ is constructed.
  We have a diagram 
  \begin{center}
    \begin{tikzcd}
      \mathcal{E}^{\bullet-2}\arrow{rr}{\delta^{\bullet-2}} &&   \calE^{\bullet-1} \\
    & &   \calE^{\bullet-1}\arrow[equal]{u}{\id} \arrow[]{llu}{k^{\bullet-1}}\arrow{rr}{\delta^{\bullet-1}}\arrow[swap]{rrd}{\partial^{\bullet-1}\circ\varphi^{\bullet-1}} & & \mathcal{E}^{\bullet} \arrow[swap]{llu}{k^{\bullet}}\\
  && && \mathcal{F}^{\bullet}\,,
    \end{tikzcd}    
  \end{center}
  where now 
  $k^{\bullet} \circ \delta^{\bullet-1} + \delta^{\bullet-2} \circ k^{\bullet-1} =  \id_{\mathcal{E}^{\bullet-1} } $. Let $v\in \Ker(\delta ^{\bullet-1}_x)$. Then
  $$v=k^{\bullet}_x \circ \delta_x^{\bullet-1}  (v)+ \delta_x^{\bullet-2} \circ k_x^{\bullet-1}  (v) = \delta_x^{\bullet-2} \circ k_x^{\bullet-1}  (v)$$
  hence 
  $$\partial_x ^{\bullet-1}\circ\varphi_x^{\bullet-1} (v)= \partial_x ^{\bullet-1}\circ\varphi_x^{\bullet-1}  \circ 
  \delta_x^{\bullet-2} \circ k_x^{\bullet-1}  (v)=
  \partial_x ^{\bullet-1} \circ \partial_x^{\bullet-2}  \circ \varphi_x ^{\bullet-2}\circ k_x^{\bullet-1}(v) = 0\,,$$
  that is $v\in \Ker (\partial_x ^{\bullet-1}\circ\varphi_x^{\bullet-1} )$.
  The inclusion $\Ker(\delta^{\bullet-1}) \subset \Ker(\partial^{\bullet-1}\circ\varphi^{\bullet-1})$
  allows to exploit relative injectivity of $\mathcal{E}^\bullet$ as follows. 
  Consider the following diagram 
  \begin{center}
    \begin{tikzcd}
      \calE^{\bullet-1}\arrow{ddrr}[swap]{p^{\bullet-1}}\arrow{rrrrdd}{\delta^{\bullet-1}}\arrow[bend right=45]{rrdddd}[swap]{\partial^{\bullet-1}\circ\varphi^{\bullet-1}}&&&& \\
      &&&&\\
     && \overline{\calE^{\bullet-1}}\arrow{rr}[swap]{\overline{\delta^{\bullet-1}}}\arrow[swap]{dd}{\overline{\partial^{\bullet-1}\circ\varphi^{\bullet-1}}} & & \mathcal{E}^{\bullet}\arrow[bend right=20]{lllluu}[swap]{k^{\bullet}}\\
     &&&&\\
   &&\mathcal{F}^{\bullet}&& \,,
    \end{tikzcd}    
  \end{center}
  where 
  $\overline{\mathcal{E}^{\bullet-1}}$ is the quotient of the bundle $\mathcal{E}^{\bullet-1}$ by $\Ker(\delta^{\bullet-1})$ and $p^{\bullet-1}:\mathcal{E}^{\bullet-1}\to\overline{\mathcal{E}^{\bullet-1}}$ is the projection to the quotient. The map $\overline{\delta^{\bullet-1}}$, respectively $\overline{\partial^{\bullet-1}\circ\varphi^{\bullet-1}}$, is the unique morphism making the middle, respectively lower, triangle commutative. 
  Thanks to the equality
  $$\delta^{\bullet-1}\circ k^{\bullet}\circ \delta^{\bullet-1}=\delta^{\bullet-1}\,,$$ the injective morphism $\overline{\delta^{\bullet-1}}$ is admissible. In fact, since $p^{\bullet-1}$ is surjective, we have the following implication
  $$\overline{\delta^{\bullet-1}}\circ p^{\bullet-1}\circ k^{\bullet}\circ \overline{\delta^{\bullet-1}}\circ p^{\bullet-1}=\overline{\delta^{\bullet-1}}\circ p^{\bullet-1} \;\Rightarrow \;\overline{\delta^{\bullet-1}}\circ (p^{\bullet-1}\circ k^{\bullet})\circ \overline{\delta^{\bullet-1}}= \overline{\delta^{\bullet-1}}\,.$$
  In particular, by the injectivity of $\overline{\delta^{\bullet-1}}$, we deduce that $ (p^{\bullet-1}\circ k^{\bullet}) \circ \overline{\delta^{\bullet-1}}=\id_{\overline{\mathcal{E}^{\bullet-1}}}$.
  Since $\mathcal{F}^{\bullet}$ is relatively injective, there must exist a $\mathcal{G}$-morphism $\varphi^{\bullet}:\mathcal{E}^{\bullet}\to \mathcal{F}^{\bullet}$ such that 
  $\lVert \varphi^{\bullet} \rVert_{\infty}\leq \rVert \overline{\partial^{\bullet-1}\circ\varphi^{\bullet-1}}\rVert_{\infty}$
  satisfying 
  $$\varphi^{\bullet}\circ \overline{\delta^{\bullet-1}}=\overline{\partial^{\bullet-1}\circ\varphi^{\bullet-1}}\,.$$ 
Thus, by commutativity of the diagram, it holds that
  $$\varphi^{\bullet}\circ \delta^{\bullet-1}=\partial^{\bullet-1}\circ\varphi^{\bullet-1}\,,$$
  and
 $$\lVert \varphi^{\bullet}\rVert_{\infty}\leq \lVert  \partial^{\bullet-1}\circ\varphi^{\bullet-1}\rVert_{\infty}.$$
This concludes the proof about the existence of an extension. 

For the uniqueness, it is sufficient to prove that any extension $\varphi^{\bullet}:\mathcal{E}^{\bullet}\to \mathcal{F}^{\bullet}$ of the trivial morphism $0:\mathcal{E}\to \mathcal{F}$ is homotopic to the trivial extension $0^{\bullet}:\mathcal{E}^{\bullet}\to \mathcal{F}^{\bullet}$. To this end, we set 
$h^{-1},h^0$ to be zero.

Assuming that $h^{\bullet-1}$ has been constructed, we look for $h^\bullet$. We consider the following diagram
\begin{center}
    \begin{tikzcd}
  \calE^{\bullet-1} \arrow{rr}{\delta^{\bullet-1}}\arrow[swap]{rd}{\varphi^{\bullet-1}-\partial^{\bullet-2} \circ h^{\bullet-1}} && \mathcal{E}^\bullet \arrow[dotted]{dl}{?}\\
  & \mathcal{F}^{\bullet-1}\,.&
    \end{tikzcd}    
  \end{center}
This is a generalized extension problem (notice that $\delta^{\bullet-1}$ is admissible because the complex is strong).  We need to check that $\mathrm{Ker}(\delta^{\bullet-1})$ is contained in the subbundle $\mathrm{Ker}(\varphi^{\bullet-1}-\partial^{\bullet-2} \circ h^{\bullet-1})$. Given $v \in \mathrm{Ker}(\delta^{\bullet-1}_x)$, by the inductive hypothesis it holds that
\begin{align*}
\partial^{\bullet-2}_x \circ h^{\bullet-1}_x \circ \delta^{\bullet-2}_x \circ k^{\bullet-1}_x (v) &=\partial^{\bullet-2}_x \circ (\varphi_x^{\bullet-2}-\partial^{\bullet-3}_x \circ h^{\bullet-2}_x) \circ k^{\bullet-1}_x(v)\\
&=\partial^{\bullet-2}_x \circ \varphi^{\bullet-2}_x \circ k^{\bullet-1}_x(v). 
\end{align*} 
We have already seen that $v=\delta^{\bullet-2}_x \circ k^{\bullet-1}_x(v)$. As a consequence
\begin{align*}
(\varphi^{\bullet-1}_x-\partial^{\bullet-2}_x \circ h^{\bullet-1}_x)(v)&=(\varphi^{\bullet-1}_x-\partial^{\bullet-2}_x \circ h^{\bullet-1}_x)(\delta^{\bullet-2}_x \circ k^{\bullet-1}_x)(v)\\
&=(\partial_x^{\bullet-2} \circ \varphi^{\bullet-2}_x \circ k^{\bullet-1}_x)(v)-(\partial^{\bullet-2}_x \circ h^{\bullet-1}_x \circ \delta^{\bullet-2} \circ k^{\bullet-1}_x)(v)\\
&=0,
\end{align*}
which proves the desired inclusion. By the relative injectivity of $\mathcal{F}^{\bullet-1}$ we can conclude. 
\end{proof}

  \begin{cor}\label{corollary_unique_resolution}
  Let $\calG$ be a measured groupoid and $\calE$ be a measurable $\calG$-bundle.
  Then there exists a unique (up to  chain homotopy) strong resolution of $\mathcal{E}$ by relatively injective bundles.
  \end{cor}
  \begin{proof}
    Given two such resolutions $\mathcal{E}^{\bullet}$ and $\mathcal{F}^{\bullet}$, it is sufficient to apply twice Lemma \ref{lemma_fundamental_lemma} exchanging the role of $\mathcal{E}^{\bullet}$ and $\mathcal{F}^{\bullet}$, and then exploit the uniqueness of the extension up to chain homotopy. 
  \end{proof}

  Corollary \ref{corollary_unique_resolution} together with Remark \ref{remark_induced_complex} implies the following
  \begin{cor}\label{corollary_cohomology}
    Let $\calG$ be a measured groupoid and $\calE$ be a measurable $\calG$-bundle.
    Let $(\mathcal{E}^{\bullet},\delta^{\bullet})$ and $(\mathcal{F}^{\bullet},\partial^{\bullet})$ be strong resolutions of $\mathcal{E}$ by relatively injective bundles.
    Then there exists a canonical isomorphism
    $$\Hm^k(\Linf(X,\mathcal{E}^{\bullet})^{\mathcal{G}})\cong \Hm^k(\Linf(X,\mathcal{F}^{\bullet})^{\mathcal{G}})$$
    for any $k\geq 0$.
    \end{cor}

\section{The bounded cohomology of measured groupoids}\label{section:bounded:cohomology}
The goal of this section is to give the definition of bounded cohomology of a $t$-discrete measured groupoid $\mathcal{G}$ with coefficients into a dual measurable bundle of Banach spaces. This is done via a characterization that follows the classic approach to bounded cohomology of groups due to Ivanov \cite{Ivanov} and to Monod \cite{monod:libro}. 
The results of the previous section shows that any two strong resolutions by relatively injective bundles of a given $\mathcal{G}$-bundle $\mathcal{E}$ share the same cohomology (Corollary \ref{corollary_cohomology}). 
Among all such resolutions, the natural candidate to define the bounded cohomology of $\mathcal{G}$ with coefficients in $\mathcal{E}$ is induced by the complex
$$(\Linf(X,\mathcal{L}(\mathcal{G}^{(\bullet+1)},\mathcal{E})),d^{\bullet+1})\,,\;\;\; 
\mathcal{L}(\mathcal{G}^{(\bullet+1)},\mathcal{E}): x\to \Linfw((\mathcal{G}^{(\bullet+1)},\nu_x^{(\bullet+1)}),E_x)$$
defined in Example \ref{example_bundle_linf}. Unfortunately, we are able to prove the relative injectivity of those bundles only when the groupoid is $t$-discrete.
For this reason, in what follows we restrict to $t$-discrete groupoids.

Let $(\mathcal{G},\nu)$ be a $t$-discrete groupoid with unit space $(X,\mu)$, where $\nu$ disintegrates with respect to the target map $t:\mathcal{G}\to X$ as 
  $$\nu=\int_X\nu^xd\mu(x)\,.$$
  For every $x\in X$, the measure $\nu^x$ is a probability measure equivalent to the counting measure on the fiber $t^{-1}(x)$. 
  We consider a $\mathcal{G}$-space $(S,\tau)$ where $\tau$ disintegrates as 
    $$\tau=\int_X\tau^xd\mu(x)\,,$$
    with respect to the map $t_{S}:S\to X$.
Let $\mathcal{E}:x\to E_x$ be a measurable $\mathcal{G}$-bundle over $X$ with left action $L$ which is the dual of a separable $\mathcal{G}$-bundle $\mathcal{F}$. 
We denote by $\mathcal{F}^{\bullet}\coloneqq (t^{\bullet}_S)^*\mathcal{F}$, by $\mathcal{E}^{\bullet}$ its dual and by
$\mathcal{L}(S^{(\bullet+1)},\mathcal{E})$ the $\mathcal{G}$-bundle over $X$ whose fiber is $\Linfw((S^{(\bullet+1)},\tau_x^{(\bullet+1)}),E_x)$, for $x\in X$.
Here the $\mathcal{G}$-action works as follows
\begin{equation}\label{equation_action}
   \overline{L}(g) \lambda (s_0,\ldots, s_{\bullet})= L(g) \lambda(g^{-1}s_0,\ldots,g^{-1}s_{\bullet})
\end{equation}
for every $g\in \mathcal{G}$, $\lambda\in\Linfw((S^{(\bullet+1)},\tau_{s(g)}^{(\bullet+1)}),E_{s(g)})$. 
Similarly, we define the predual bundle 
$\mathcal{L}^{\flat}(S^{(\bullet+1)}, \mathcal{F})$.
The same arguments used to show the isomorphism of Theorem \ref{proposition_disintegration_isomorphism} can be applied to prove the $\mathcal{G}$-equivariant canonical isometric isomorphism
\begin{equation}\label{equation_disintegration_lone_S}
  \Lone(S^{(\bullet+1)},\mathcal{F}^{\bullet+1})\cong \Lone(X,\mathcal{L}^{\flat}(S^{(\bullet+1)},\mathcal{F}))\,.
\end{equation}
whose dual version is
\begin{equation}\label{equation_disintegration_linf_S}
  \Linf(S^{(\bullet+1)},\mathcal{E}^{\bullet+1})\cong \Linf(X,\mathcal{L}(S^{(\bullet+1)},\mathcal{E}))\,
\end{equation}

We introduce the coboundary operator 
$$d^{\bullet}: \mathcal{L}(S^{(\bullet+1)},\mathcal{E}) \rightarrow \mathcal{L}(S^{(\bullet+2)},\mathcal{E})\,,\;\;\; d^{\bullet}=\sum\limits_{i=0}^{\bullet+1} (-1)^i \delta_i^{\bullet}$$
where 
\begin{gather}\label{equation_delta}
  \delta_i^{\bullet}: \mathcal{L}(S^{(\bullet+1)},\mathcal{E})\to \mathcal{L}(S^{(\bullet+2)},\mathcal{E}) \\ \notag(\delta_i^{\bullet})_x f (s_0,\ldots,s_{\bullet+1})=f(s_0,\ldots,\widehat{s_{i}},\ldots,s_{\bullet+1})
\end{gather}
for every $f\in \Linfw((S^{(\bullet+1)},\tau_x^{(\bullet+1)}),E_x)$.

\begin{lemma}
 In the context described so far, the function $d^{\bullet}$ is a morphism of $\mathcal{G}$-bundles.
\end{lemma}
\begin{proof}
We show that each $\delta_i^{\bullet}$ sends measurable sections of $\mathcal{L}(S^{(\bullet+1)},\mathcal{E})$ to measurable sections of $\mathcal{L}(S^{(\bullet+2)},\mathcal{E})$. We pick any section $\sigma$ of $\mathcal{L}(S^{(\bullet+1)},\mathcal{E})$ and a measurable section 
$\eta$ of $\mathcal{L}^{\flat}(S^{(\bullet+2)},\mathcal{F})$ and we prove that
$$x\mapsto \langle(\delta_i^{\bullet})_x\sigma(x) , \eta(x)\rangle$$ is $\mu$-measurable. 

 It holds that 
\begin{align*}
  &\langle\delta_i^{\bullet}\sigma(x),\eta(x)\rangle\\
=& \int_{S^{(\bullet+2)}}
   \langle\sigma(x)(s_0,\ldots,\widehat{s}_i, \ldots, s_{\bullet+1})\,,\, \eta(x)(s_0,\ldots, s_{\bullet+1})\rangle
  d\tau^{\bullet+2}_x (s_0,\ldots, s_{\bullet+1})\,\\
  =& \int_{S}\cdots \int_{S}
  \langle\sigma(x)(s_0,\ldots,\widehat{s}_i, 
   \ldots, s_{\bullet+1})\, , \int_S^B \eta(x)(s_0,\ldots , s_{\bullet+1}) d\tau^x(s_i) \rangle
 d\tau^x (s_0)\cdots d\tau^x(s_{\bullet+1})\,,
\end{align*}
where $\int^B_S$ denotes the Bochner integral. In the above equation we applied the fact that the Bochner integral commutes with linear operators \cite[Section 3.1]{monod:libro}. 
We notice that the map 
$$(s_0,\ldots,s_{i-1},s_{i+1},\ldots,s_{\bullet+1})\mapsto 
\int_S^B \eta(x)(s_0,\ldots , s_{\bullet+1}) d\tau^x(s_i)$$
is a measurable section of $\mathcal{L}^{\flat}(S^{(\bullet+1)},\mathcal{F})$. Indeed, the Bochner integral is fiberwise continuous and measurability follows by Remark \ref{oss morphism dense subset} applied to the generating family defined in the proof of Theorem \ref{proposition_disintegration_isomorphism}. The duality between $\mathcal{L}^{\flat}(S^{(\bullet+1)},\mathcal{F})$ and $\mathcal{L}(S^{(\bullet+1)},\mathcal{E})$ implies that the map $x\mapsto \langle\delta_i^{\bullet}\sigma(x),\eta(x)\rangle$ is $\mu$-measurable, as claimed. 


Furthermore, $d^{\bullet}$ is bounded since each $\delta^{\bullet}_i$ has norm at most $\bullet+1$ by definition. Finally, a straightforward computation show that each $\delta^{\bullet}_i$ is also $\mathcal{G}$-equivariant, whence $d^{\bullet}$ is a morphism of $\mathcal{G}$-bundles. 
\end{proof}

\begin{rec_defn}[\ref{def bounded cohomology}]
  The \emph{bounded cohomology} $\Hmb^{\bullet}(\mathcal{G},\mathcal{E})$ of a $t$-discrete measured groupoid $\mathcal{G}$ with coefficients in a dual measurable $\mathcal{G}$-bundle $\mathcal{E}$
 is the cohomology of the complex $(\Linf(X, \mathcal{L}(\mathcal{G}^{(\bullet+1)},\mathcal{E}))^{\mathcal{G}}, d^{\bullet})$, namely
$$\Hmb^{k}(\mathcal{G},\mathcal{E})\coloneqq \Hm^k(\Linf(X, \mathcal{L}(\mathcal{G}^{(\bullet+1)},\mathcal{E}))^{\mathcal{G}}, d^{\bullet})\,.$$
\end{rec_defn}

\begin{oss}
  The above definition extends the one of \emph{measurable bounded cohomology} of a measured groupoid introduced in \cite{sarti:savini:23}. Indeed, if $\mathcal{E}$ is the constant bundle, namely each fiber coincides with a fixed Banach space $E$, essentially bounded sections of $\mathcal{E}^{\bullet}$ are precisely $E$-valued essentially bounded functions on $\mathcal{G}^{(\bullet+1)}$. Thus, thanks to the isomorphism of Theorem \ref{proposition_disintegration_isomorphism}, we have 
  $$\Linf(X, \mathcal{L}(\mathcal{G}^{(\bullet+1)},\mathcal{E}))\cong \Linfw(\mathcal{G}^{(\bullet+1)},E)\,.$$
and these modules are the ones used in \cite{sarti:savini:23} to define the functor $\mathrm{H}^\bullet_{\text{mb}}$ with constant coefficients. 
\end{oss} 

\begin{oss}
The bounded cohomology of a measured groupoid with coefficients into the dual of a separable bundle was already mentioned by Delaroche and Renault \cite[Section 4.3]{delaroche:renault}. The authors consider \emph{bounded inhomogeneous 1-cocycles} and \emph{1-coboundaries} in order to characterize the amenability of semidirect groupoids in terms of the vanishing of the resulting 1-cohomology. As it happens for groups \cite{miolibro,monod:libro}, we believe that their definition via inhomogeneous cochains coincides with our homogeneous construction. 
\end{oss}

In analogy with bounded cohomology of groups \cite[Proposition 7.4.1]{monod:libro}, when the coefficient bundle is relatively injective the bounded cohomology vanishes. 
\begin{prop}
  Let $\calG$ be a $t$-discrete measured groupoid and $\mathcal{E}$ a measurable $\calG$-bundle that is the dual of a separable measurable $\calG$-bundle.
  If $\mathcal{E}$ is relatively injective, then 
  $$ \Hmb^{k}(\calG,\mathcal{E})=0\,$$
  for every $k\geq 1$. 
\end{prop}

\begin{proof}
  Consider the following complex
  \begin{equation}\label{equation resolution}
    \ldots 0\longrightarrow \mathcal{E}\longrightarrow \mathcal{E}\longrightarrow 0\longrightarrow \ldots  
  \end{equation}
  which is an augmented resolution of $\mathcal{E}$ by relatively injective modules. We claim that it is strong. If this is true, the conclusion follows by Corollary \ref{corollary_cohomology}. Indeed, the bounded cohomology of $\mathcal{G}$ coincides with the one of the complex 
  $$0\longrightarrow \Linf(X,\mathcal{E})^{\mathcal{G}}\longrightarrow 0\longrightarrow 0\longrightarrow \ldots\,,$$ which is zero in all positive degrees. 
  To prove the claim, we construct the following contracting homotopy for the resolution of Equation \eqref{equation resolution}
  $$k^\bullet =\begin{cases}
    0 & \text{ if } \bullet\neq 0\\
    \Id & \text{ if } \bullet= 0
  \end{cases}\,.$$
  The fact that each $k^{\bullet}$ does not increase the norm is obvious. To verify that they are contracting homotopies is straightforward. 
\end{proof}

The goal of the remaining part of this section is to prove that, for amenable $\mathcal{G}$-spaces, the complex $(\Linf(X, \mathcal{L}(S^{(\bullet+1)},\mathcal{E}))^{\mathcal{G}}, d^{\bullet})$ forms a strong resolution by relatively injective bundles. This, together with the results of Section \ref{section:homological}, shows that one can compute the bounded cohomology via essentially bounded functions on any amenable space, generalizing a classic result by Burger and Monod \cite[Theorem 2]{burger2:articolo}.

First of all, we show that the complex $(\mathcal{L}(S^{(\bullet+1)},\mathcal{E}),d^{\bullet})$ with augmentation morphism the inclusion of constants $\mathcal{E}\hookrightarrow \mathcal{L}(S,\mathcal{E})$ provides a  resolution of $\mathcal{E}$ that admits a contracting homotopy. This is true for all measured groupoids (not necessarily discrete) and all $\mathcal{G}$-spaces (not necessarily amenable). 

\begin{prop}\label{proposition_strong}
  Let $\mathcal{G}$ be a measured groupoid, $(S,\tau)$ a $\mathcal{G}$-space and $\mathcal{E}$ a $\mathcal{G}$-bundle that is the dual of a separable $\mathcal{G}$-bundle $\mathcal{F}$.
  Then the resolution
  \begin{center}
    \begin{tikzcd}
      0 \arrow{r}& \mathcal{E}\arrow{r}{\epsilon}&\mathcal{L}(S,\mathcal{E})\arrow{r}{\delta^{0}}&\mathcal{L}(S^{(2)},\mathcal{E})\arrow{r}&\cdots
    \end{tikzcd}
  \end{center}
  where the augmentation $\calG$-morphism $\epsilon$ is defined as
  $$\epsilon :\mathcal{E}\to \mathcal{L}(S,\mathcal{E})\,,\;\;\; \epsilon_x(v)(s)\coloneqq v $$
  is strong. 
  \end{prop}
  
  \begin{proof}
    Given any section $\lambda$ of $\mathcal{L}(S^{(\bullet+1)},\mathcal{E})$ we denote by $\lambda^x\coloneqq \lambda(x)$.
    We set
    $$
    k^0:\mathcal{L}(S,\mathcal{E})\to \mathcal{E}\,,\;\;\;
    k_x^0 (\lambda^x )\coloneqq \int^{(GD)}_{S} \lambda^x(s) d\tau^x(s)\,
    $$
 and in higher degree we set
$$
   k^{\bullet}:\mathcal{L}(S^{(\bullet+1)},\mathcal{E})\to \mathcal{L}(S^{(\bullet)},\mathcal{E})\,,\;\;\;
$$
$$
      k_x^{\bullet}( \lambda^x) (s_0,\ldots,s_{\bullet-1})\coloneqq \int^{(GD)}_{S} \lambda^x(s,s_0,\ldots,s_{\bullet-1}) d\tau^x(s)\,
$$
  where $\int^{(GD)}$ is the Gelfand-Dunford integral \cite[Section 3.2]{monod:libro}. 
 To see that $k^0$ is a morphism in the sense of Definition \ref{definition_morphism_bundles}, we take a measurable section $x\mapsto\lambda^x$ of the bundle $\mathcal{L}(S,\mathcal{E})$ and we show that the map $x\mapsto \kappa_x^0(\lambda^x)$ is a measurable section of $\mathcal{E}$. Since $\mathcal{E}$ is the dual of $\mathcal{F}$, we need to show that for any measurable section $\eta$ of $\mathcal{F}$ the map 
  $$x\mapsto \langle k_x^0 (\lambda^x), \eta(x) \rangle =\langle \int^{(GD)}_{S} \lambda^x(s) d\tau^x(s), \eta(x) \rangle$$ is 
  $\mu$-measurable, where $\langle \cdot,\cdot\rangle$ is the duality pairing between $E_x$ and $F_x$. This is equivalent to show the $\mu$-measurability of 
\begin{equation}\label{eq Gelfand pairing}
x\mapsto \int_{S} \langle\lambda^x(s), \eta(x) \rangle d\tau^x(s)\,.
\end{equation}
  Since essentially bounded sections are dense, by Remark \ref{oss morphism dense subset} we can suppose that $\lambda$ is essentially bounded. Thus we can think of $\lambda$ as a section on $S$ and we can rewrite Equation \eqref{eq Gelfand pairing} as follows
$$
x \mapsto \int_S \langle \lambda(s),\eta(t_S(s)) \rangle d\tau^x(s)\,. 
$$
The measurability follows again by the disintegration properties of the family $\{\tau^x\}_{x \in X}$. 
  
In higher degree, we can consider two sections $\lambda\in \Linf(X,\mathcal{L}(S^{(\bullet+1)},\mathcal{E}))$ and $\eta\in \Lone(X,\mathcal{L}^{\flat}(S^{(\bullet)},\mathcal{F}))$. We need to show that 
  $$x\mapsto \langle k_x^{\bullet+1} (\lambda^x), \eta^x \rangle =\langle \int^{(GD)}_{S} \lambda^x(s,s_0,\ldots,s_{\bullet-1}) d\tau^x(s)\,,\, \eta^x(s_0,\ldots,s_{\bullet-1}) \rangle\,$$
  is $\mu$-measurable.
  Here the pairing $\langle \cdot,\cdot\rangle$ on the left is the one between $\Lone((S^{(\bullet)},\tau^{\bullet}_x), F_x)$ with $\Linfw((S^{(\bullet)},\tau^{\bullet}_x),E_x)$, whereas the one on the right refers to the duality between $F_x$ and $E_x$. 
  By the isomorphism of Equation \eqref{equation_disintegration_lone_S}, we can think of both $\lambda$ and $\eta$ as sections defined over $S^{(\bullet+1)}$. In the case of $\eta$, we can actually extend it by the formula
  $$\widehat{\eta}(s,s_0,\ldots,s_{\bullet-1})\coloneqq\eta(s_0,\ldots,s_{\bullet-1})\,.$$
As a consequence, we obtain the following computation
  \begin{align*}
    &\langle \int^{(GD)}_{S} \lambda^x(s,s_0,\ldots,s_{\bullet-1}) d\tau^x(s)\,,\, \eta^x(s_0,\ldots,s_{\bullet-1}) \rangle\\
  =\;& \int_{S} \langle \lambda(s,s_0,\ldots,s_{\bullet-1})\,,\,\widehat{\eta}(s,s_0,\ldots,s_{\bullet-1}) \rangle d\tau^x(s)\,.
  \end{align*}
For any fixed $(s_0,\ldots,s_{\bullet-1})$, the measurability of the map 
  $$s\mapsto \langle \lambda(s,s_0,\ldots,s_{\bullet-1})\ , \ \widehat{\eta}(s,s_0,\ldots,s_{\bullet-1}) \rangle$$
implies that   
$$x\mapsto \int_{S} \langle \lambda(s,s_0,\ldots,s_{\bullet-1})\,,\,\widehat{\eta}(s,s_0,\ldots,s_{\bullet-1}) \rangle d\tau^x(s)$$ is measurable, again by the usual disintegration properties of the family $\{\tau^x\}_{x \in X}$. The boundedness of $k^\bullet$ is straightforward, thus we can conclude that $k^{\bullet}$ is a morphism. 
  
Finally the equality
  $$\delta_x^{\bullet-1}\circ k_x^{\bullet}(\lambda^x) +k_x^{\bullet+1}\circ \delta_x^{\bullet}(\lambda^x) =\lambda^x$$ 
  holds for every $x\in X$.
  In fact, the standard algebraic argument used for groups can be similarly reproduced fiberwise in this context. This concludes the proof.
  \end{proof}

We move on to the proof that, whenever $\mathcal{G}$ is $t$-discrete and $S$ is $\mathcal{G}$-amenable, each bundle $\mathcal{L}(S^{\bullet},\mathcal{E})$ is relatively injective.
The strategy is the same used by Monod in the case of groups \cite{monod:libro}.
Indeed, we first prove relative injectivity of the above bundle when $S=\mathcal{G}$.

\begin{prop}\label{proposition_relative_injective1}
  Let $\mathcal{G}$ be a $t$-discrete measured groupoid and $\mathcal{E}$ be a separable measurable $\mathcal{G}$-bundle. Then the $\calG$-bundle $\mathcal{L}(\mathcal{G}^{(\bullet)},\mathcal{E}): x \mapsto \mathrm{L}^\infty_{\textup{w}^\ast}((\mathcal{G}^{(\bullet)},\nu_x^{(\bullet)}),E_x)$ is relatively injective.
\end{prop}

\begin{proof}

We will only deal with the degree one case, since the proof for higher degrees is analogous and we leave it to the reader. 

By hypothesis, $\mathcal{G}^x$ is discrete for every $x\in X$. Hence the space $\Linfw((\calG,\nu^x),E_x)$ boils down to $\ell^{\infty}(\mathcal{G}^x,E_x)$, that is the Banach space of $E$-valued bounded functions on $\mathcal{G}^x$ endowed with the supremum norm. In particular, the equivalence relation on functions is trivial, since the latter is induced by the counting measure.

Consider now the extension problem 
\begin{center}
  \begin{tikzcd}
    \mathcal{F}\arrow{rr}{\alpha} \arrow{rd}[swap]{\beta}    &&\mathcal{K}\arrow[dotted]{ld}{?} \arrow[bend right=30]{ll}[swap]{\sigma} \\
    &\mathcal{L}(\mathcal{G},\mathcal{E}) &\,
  \end{tikzcd}
\end{center}
where
$\mathcal{F},\calK$ are measurable $\calG$-bundles of Banach spaces,
$$\alpha:\mathcal{F}\to\calK$$ is an admissible $\calG$-morphism and 
$$\beta:\calF\to\mathcal{L}(\mathcal{G},\mathcal{E}) $$ is a
$\calG$-morphism.
Moreover, 
$$\sigma:\calK\to\calF$$ is a left inverse
for $\alpha$ such that $||\sigma||_{\infty}\leq 1$. In other words,  
\begin{equation}\label{equation_inverse}
\sigma_x\circ \alpha_x=\id_{F_x}
\end{equation}
holds for almost every $x\in X$.
By definition we have 
\begin{equation}\label{equation_alpha}
\alpha_{t(g)} (L_{\mathcal{F}}(g) v)=L_{\mathcal{K}}(g)\alpha_{s(g)}( v)
\end{equation}
for almost every $g\in \calG$ and for every $v\in \calF$ and 
\begin{equation}\label{equation_beta}
 \beta_{t(g)} (L_{\mathcal{F}}(g) v)(h)=L_{\mathcal{E}}(g) \beta_{s(g)} ( v) (g^{-1}h)
\end{equation}
for almost every $g,h \in \calG^x$ and for every $v\in \calF$. 
We define a morphism
$$\psi:\calK\to\mathcal{L}(\mathcal{G},\mathcal{E}) $$ as
$$\psi_x(w)(h)\coloneqq \beta_x(L_{\mathcal{F}}(h) \sigma_{s(h)} (L_{\mathcal{K}}(h^{-1}) w ))(h )$$
for every $x\in X$, $h\in \calG^x$ and $w\in \calK_x$.

The fact that $\psi$ is a morphism of $\mathcal{G}$-bundles follows by the fact that both $\sigma$ and $\beta$ are so. 
Moreover,
for every $g,h\in \calG^x$ and for every $w\in \calK_{s(g)}$ we have 
\begin{align*}
\psi_{t(h)}(L_{\mathcal{K}}(g) w)(h)
=\;&  \beta_{t(h)}(L_{\mathcal{F}}(h) \sigma_{s(h)} (L_{\mathcal{K}}(h^{-1})L_{\mathcal{K}} (g) w ) )(h )\\
=\;&  L_{\mathcal{E}}(g) \beta_{s(g)}(L_{\mathcal{F}}(g^{-1}) L_{\mathcal{F}}(h) \sigma_{s(h)} (L_{\mathcal{K}}(h^{-1} g) w )) (g^{-1}h )\\
=\;&  L_{\mathcal{E}}(g) \beta_{s(g)}(L_{\mathcal{F}}(g^{-1}h) \sigma_{s(h)} (L_{\mathcal{K}}(h^{-1} g)(w)))(g^{-1}h )\\
=\;& (\overline{L}(g) (\psi_{s(g)}))(w)(h)\,,
\end{align*}
hence $\psi$ is a $\calG$-morphism.
In the above computation we moved from the second line to the third one thanks to Equation \eqref{equation_beta}, then we exploited the properties of the $\mathcal{G}$-actions, and we concluded by definition of $\psi$  and of the $\calG$-action on $\mathcal{L}(S,\mathcal{E})$ given by Equation \eqref{equation_action}.
 
Furthermore, for almost every $g\in \calG$ and for every $v\in F_{t(g)}$ we have
 \begin{align*}
\psi_{t(h)}\circ \alpha_{t(h)} (v)(h)
=\;&  \beta_{t(h)}(L_{\mathcal{F}}(h)\sigma_{s(h)} (L_{\mathcal{K}}(h^{-1}) \alpha_{t(h)} (v) ) )(h )\\
=\;&  \beta_{t(h)}(L_{\mathcal{F}}(h)\sigma_{s(h)}  \alpha_{s(h)}(L_{\mathcal{F}}(h^{-1}) v) )(h )\\
=\;&  \beta_{t(h)}(L_{\mathcal{F}}(h)L_{\mathcal{F}}(h^{-1}) v )(h )\\
=\;&  \beta_{t(h)}(v )(h )\,,
\end{align*}
hence $\psi_x\circ \alpha_x=\beta_x$ 
for almost every $x\in X$.
In the above computation we moved from the first line to the second one by definition of $\psi$, from the second line to the third one exploiting Equation \eqref{equation_alpha}, from the third line to the fourth one by Equation \eqref{equation_inverse} and we concluded thanks to the property of the $\mathcal{G}$-action.

Finally, since $\lVert \sigma \rVert_{\infty}\leq 1$, for almost every $x\in X$ and for every $g\in \calG^x$, $w\in K_{x}$, we have
$$
\lVert \beta_{x}(L_{\mathcal{F}}(h) \sigma_{s(h)} (L_{\mathcal{K}}(h^{-1}) w))(h) \rVert_{E_x} \leq  \lVert \beta_{x}\rVert_{\infty} \lVert w \rVert_{K_x}
$$
and thus 
$$\lVert \psi \rVert_{\infty} \leq \lVert \beta \rVert_{\infty} \,.$$
This concludes the proof.
\end{proof}

\begin{oss}
We stress the fact that the above argument cannot be extended to general measured groupoids. Indeed, the discreteness assumption allows to evaluate an essentially bounded section of $\mathcal{L}(\mathcal{G},\mathcal{E})$ on each fiber, which is crucial to define the solution of an extension problem . 

The strategy adopted by Monod \cite[Lemma 4.4.3]{monod:libro} to prove the same result for topological groups strongly relies on topological considerations. For this reason, we do not see how to adapt it in our context. 
\end{oss}

\begin{prop}\label{proposition_relative_injective}
  Let $\mathcal{G}$ be a $t$-discrete measured groupoid, $(S,\tau)$ an amenable $\mathcal{G}$-space and $\mathcal{E}$ be a $\mathcal{G}$-bundle that is the dual of a separable $\mathcal{G}$-bundle $\mathcal{F}$. Then the $\calG$-bundle $\mathcal{L}(S^{(\bullet+1)},\mathcal{E})$ is relatively injective for every $\bullet\geq 0$.
\end{prop}

\begin{proof}
Thanks to Lemma \ref{lemma:rel:inj1}, it is sufficient to prove the proposition for $\mathcal{L}(S, \mathcal{E})$. Indeed, by \cite[Corollary 2.3.3]{monod:libro} for each $x\in X$ we have canonical isometric isomorphisms 
$$\Linfw((S^x)^{\bullet+2}, E_x))\cong \Linfw(S^x, (\Linfw((S^x)^{\bullet+1},E_x)))$$
that induce a bundle isomorphism
$$\mathcal{L}(S^{(\bullet+1)}, \mathcal{E})\cong \mathcal{L}(S, \mathcal{L}(S^{(\bullet)},\mathcal{E}))\,,$$
in any degree.

  Moreover, by \cite[Corollary 2.3.3]{monod:libro} we also have a canonical isometric isomorphism
   $$\Linfw(S^x \times \mathcal{G}^x, E_x)\cong \Linfw(\mathcal{G}^x,\Linfw(S^x,E_x))\,$$
   for every $x\in X$.
In this way we obtain an isomorphism of $\mathcal{G}$-bundles
$$\mathcal{L}(S \ast \mathcal{G}, \mathcal{E})\cong \mathcal{L}(\mathcal{G}, \mathcal{L}(S,\mathcal{E}))\,.$$
Thanks to Proposition \ref{proposition_relative_injective1} the right hand side is relatively injective, thus the same holds for the left hand side.

Consider now the morphism
$$\iota:\mathcal{L}(S,\mathcal{E})\to \mathcal{L}(S \ast \mathcal{G}, \mathcal{E})$$ induced by inclusions $\Linfw(S^x,E_x)\hookrightarrow \Linfw(S^x \times G^x, E_x)$. By definition we see  
$\lVert \iota \rVert_{\infty}\leq 1$. 
We claim that $\iota$ admits a left inverse of norm at most one. 
If this is the case, we are in the condition to apply Lemma \ref{lemma:rel:inj2}. Since $\mathcal{L}(S \ast \mathcal{G}, \mathcal{E})$ is relatively injective, we can conclude. 

The rest of the proof is devoted to prove the claim, namely to construct a morphism 
$\sigma: \mathcal{L}(S \ast \mathcal{G}, \mathcal{E})\to \mathcal{L}(S,\mathcal{E})$ such that $\sigma\circ \iota=\id$ and $\lVert \sigma \rVert_{\infty}\leq 1$.
Since $S$ is $\mathcal{G}$-amenable, there exists an equivariant family 
$$\mathfrak{m}^s: \Linf((\mathcal{G},\nu^x),E_x)\to \mathbb{R}$$
of means. We define $\sigma$ fiberwisely as 
$$
\sigma_x: \mathrm{L}^\infty_{\textup{w}^\ast}(S^x \times \mathcal{G}^x,E_x) \rightarrow \mathrm{L}^\infty_{\textup{w}^\ast}(S^x,E_x),
$$
$$\langle\sigma_x(\varphi)(s),v\rangle\coloneqq \mathfrak{m}^s(g\mapsto \langle \varphi(g,s),v  \rangle )$$
where $v\in E_x^{\flat}$. 
Here $\langle\cdot ,\cdot\rangle$ refers to the pairing between $E_x$ and $E_x^{\flat}$. 

The fact that $\sigma$ is well-defined relies on the fact that the predual $\mathcal{F}:=\mathcal{E}^\flat$ is separable. Now we show that $\sigma$ is a morphism. To this end we fix a measurable section $\eta$ of $\mathcal{L}(S \ast \mathcal{G},\mathcal{E})$ and we prove that 
$x\mapsto \sigma_x\circ \eta(x)$ is a section of $\mathcal{L}(S,\mathcal{E})$. 
Precisely, since $\mathcal{L}(S,\mathcal{E})$ is the dual of $\mathcal{L}^{\flat}(S,\mathcal{F})$, we need to show that 
the map $$x\mapsto \langle \sigma_x\circ \eta(x), \xi (x)\rangle$$ is $\mu$-measurable for every measurable section $\xi$ of $\mathcal{L}^{\flat}(S,\mathcal{F})$. Here $\langle\cdot,\cdot\rangle$ refers to the pairing between $\Lone(S, E_x^{\flat})$ and $\Linfw(S, E_x)$. 
By Remark \ref{oss morphism dense subset} we can reduce to the case when $\eta\in \Linf(X, \mathcal{L}(\mathcal{G}*S,\mathcal{E}))$ and $\xi\in \Lone(X,\mathcal{L}^{\flat}(S, \mathcal{F}))$. 
Thus we have 
\begin{align*}
  \langle \sigma_x\circ \eta(x), \xi (x)\rangle&= \int_S \langle\sigma_x\circ \eta(x)(s), \xi (x)(s)\rangle d\tau^x(s)\\
  &=  \int_S \mathfrak{m}^s(g\mapsto \langle \eta(g,s), \xi(s)  \rangle ) d\tau^x(s)
\end{align*}
In the above computation we exploited twice the disintegration isomorphisms of Section \ref{section disintegration} to view both $\sigma$ and $\eta$ as measurable sections on $S \ast \mathcal{G}$ and $S$, respectively. The measurability now follows since $\mathfrak{m}^s$ is a system of means, hence $s\mapsto \mathfrak{m}^s(g\mapsto \langle \eta(g,s),\xi(s) \rangle )$ is measurable, and because $\tau^x$ is a Borel system. 

We move on our proof and we focus on the $\mathcal{G}$-equivariance. 
On one hand we have that
\begin{align}\label{LHS mean}
 g((\sigma_{s(g)}(\varphi_{s(g)})(s))(v)&=\mathfrak{m}^{g^{-1}s}(h \mapsto \langle L_{\mathcal{E}}(g) \varphi_{s(g)}(h,g^{-1}s) , v \rangle)\\
&=g^{-1}\mathfrak{m}^s(h \mapsto \langle  L_{\mathcal{E}}(g)  \varphi_{s(g)}(h,g^{-1}s),v\rangle) \nonumber \\
&=\mathfrak{m}^s(h \mapsto \langle L_{\mathcal{E}}(g) \varphi_{s(g)}(g^{-1}h,g^{-1}s),v \rangle, \nonumber 
\end{align}
where we moved from the first line to the second one using the invariance of the system of means and we concluded by applying the usual action on essentially bounded sections. 

On the other hand we compute
\begin{align}\label{RHS mean}
\sigma_{t(g)}(g\varphi_{s(g)}(s))(v)&=\mathfrak{m}^s(h \mapsto \langle (g\varphi_{s(g)})(h,s),v \rangle)\\
&=\mathfrak{m}^s(h \mapsto \langle L_{\mathcal{E}}(g) \varphi_{s(g)}(g^{-1}h,g^{-1}s),v \rangle). \nonumber
\end{align}
Since Equation \eqref{LHS mean} and Equation \eqref{RHS mean} produce the same result, the claim is true and this concludes the proof. 
\end{proof}

Putting together Proposition \ref{proposition_strong} and Proposition \ref{proposition_relative_injective} we finally obtain the following
\begin{rec_thm}[\ref{thm:amenable}]
Let $\calG$ be a $t$-discrete measured groupoid, $(S,\tau)$ an amenable $\mathcal{G}$-space and $\mathcal{E}$ a measurable $\calG$-bundle that is the dual of a separable measurable $\calG$-bundle.
Then we have a natural isomorphism
$$\Hm^{k}(\Linf(X, \mathcal{L}(S^{(\bullet+1)},\mathcal{E}))^{\mathcal{G}})\cong \Hmb^{k}(\calG,\mathcal{E})\,$$
for every $k\geq 0$.
\end{rec_thm}
\begin{proof}
  Since the resolution 
  \begin{center}
    \begin{tikzcd}
      0 \arrow{r}& \mathcal{E}\arrow{r}{\epsilon}&\mathcal{L}(S,\mathcal{E})\arrow{r}{\delta}&\mathcal{L}(S^{(2)},\mathcal{E})\arrow{r}&\cdots 
    \end{tikzcd}
  \end{center}
  is strong (Proposition \ref{proposition_strong}) by relatively injective bundles (Proposition \ref{proposition_relative_injective}), we can apply Corollary \ref{corollary_cohomology} and get the desired isomorphisms. 
\end{proof}

A direct application of Theorem \ref{thm:amenable} is that bounded cohomology can be computed using strong boundaries. In the following result we use the definition of $\mathcal{G}$-boundary introduced by the authors \cite{sarti:savini:24}.

\begin{cor}\label{corollary:boundaries}
  Let $\calG$ be a $t$-discrete measured groupoid, $(B,\tau)$ a $\mathcal{G}$-boundary and $\mathcal{E}$ a measurable $\calG$-bundle that is the dual of a separable measurable $\calG$-bundle.
  Then we have a natural isomorphism
  $$\Hm^{k}(\Linf(X, \mathcal{L}(B^{\bullet},\mathcal{E}))^{\mathcal{G}})\cong \Hmb^{k}(\calG,\mathcal{E})\,$$
  for every $k\geq 0$.
  \end{cor}

Another relevant consequence is the following vanishing result, that was already proved by the authors (\cite[Theorem 4]{sarti:savini:23}) with different techniques. 

\begin{rec_cor}[\ref{cor vanishing amenable}]
Let $\mathcal{G}$ a $t$-discrete amenable measured groupoid and $\mathcal{E}$ a measurable $\calG$-bundle that is the dual of a separable measurable $\calG$-bundle. Then we have that
$$
\Hmb^{k}(\calG,\mathcal{E}) \cong 0,
$$
for $k \geq 1$. 
\end{rec_cor}

\begin{proof}
If the groupoid is amenable, the unit space $X$ is an amenable $G$-space. Additionally, the fibred product $X^{(\bullet)}$ is done with respect to the identity, hence we have that
$$
\mathrm{L}^\infty(X,\mathcal{L}(X^{(\bullet+1)},\mathcal{E}))^{\mathcal{G}} \cong \mathrm{L}^\infty(X,\mathcal{E})^{\mathcal{G}},
$$
in any degree. 
By Theorem \ref{thm:amenable} the complex
\begin{center}
    \begin{tikzcd}
      0 \arrow{r}& \mathrm{L}^\infty(X,\mathcal{E})^{\mathcal{G}} \arrow{r}{0} & \mathrm{L}^\infty(X,\mathcal{E})^{\mathcal{G}} \arrow{r}{\mathrm{Id}} & \mathrm{L}^\infty(X,\mathcal{E})^{\mathcal{G}} \arrow{r}{0} & \cdots
    \end{tikzcd}
\end{center}
computes the bounded cohomology of $\calG$ with $\mathcal{E}$-coefficients. Since it is acyclic, we obtain
$$
\Hmb^{k}(\calG,\mathcal{E}) \cong 0
$$
when $k \geq 1$. 
\end{proof}

  \bibliographystyle{alpha}

\bibliography{biblionote}

\end{document}